\documentclass[10pt,a4paper]{article}

\usepackage{amsmath}
\usepackage{amsthm}
\usepackage{amssymb}
\usepackage{mathtools}

\usepackage[hidelinks,pdfusetitle]{hyperref}

\usepackage[width=13cm,height=21cm]{geometry}
\usepackage[margin=10pt,font=small,labelfont=bf,
labelsep=endash]{caption}

\usepackage[capitalise]{cleveref}
\usepackage[backend=bibtex,url=false,isbn=false,maxcitenames=5,maxbibnames=99,sortcites]{biblatex}
\newtheorem{thm}{Theorem}
\newtheorem{proposition}{Proposition}[section]
\newtheorem{lemma}[proposition]{Lemma}
\newtheorem{cor}[proposition]{Corollary}
\newtheorem{definition}[proposition]{Definition}

\newtheorem{criterion}{Criterion}
\theoremstyle{remark}
\newtheorem{remark}[proposition]{Remark}

\numberwithin{equation}{section}

\usepackage[dvipsnames]{xcolor}

\newcommand{\conj}[1]{\overline{#1}} 

\DeclareMathOperator{\PV}{PV} 
\newcommand{\lap}{\Delta} 
\newcommand{\Ai}{{\rm Ai}}

\let\div\relax 
\DeclareMathOperator{\div}{div}

\newcommand{\T}{\mathbb{T}}
\newcommand{\R}{\mathbb{R}}
\newcommand{\C}{\mathbb{C}}
\newcommand{\N}{{\mathbb N}}
\newcommand{\Z}{{\mathbb Z}}

\newcommand{\ee}{\mathrm{e}} 
\newcommand{\ii}{\mathrm{i}} 
\newcommand{\dd}{\mathrm{d}} 

\newcommand{\OO}{{\cal O}}
\newcommand{\eps}{{\varepsilon}}

\renewcommand{\Re}{\mathcal{R}e\,}
\renewcommand{\Im}{\mathcal{I}m\,}

\newcommand{\bfu}{\mathbf{u}}
\newcommand{\bfx}{\mathbf{x}}
\newcommand{\modeU}{\mathcal{U}} 
\newcommand{\modeV}{\mathcal{V}} 

\newcommand{\Us}{U_s}
\newcommand{\ds}{\Delta_s}
\newcommand{\hs}{H_s}

\newcommand{\vv}{\phi}
\newcommand{\vc}{\vv_{\mathrm{c}}}
\newcommand{\vf}{\vv_{\mathrm{f}}}
\newcommand{\cH}{\mathcal{H}}
\newcommand{\tvv}{\tilde \vv}

\newcommand{\inviscid}{\mathrm{inv}}
\newcommand{\interior}{\mathrm{int}}
\newcommand{\bl}{\mathrm{bl}}
\newcommand{\Pinv}{\Phi_{\inviscid}}

\newcounter{savecntrP6}
\newcounter{restorecntrP6}
\newcounter{savecntrCNRS}
\newcounter{restorecntrCNRS}
\newcounter{savecntrP7}
\newcounter{restorecntrP7}

\author{Anne-Laure Dalibard\setcounter{savecntrP6}{\value{footnote}}\thanks{Sorbonne Universit\'es, UPMC Univ. Paris 06, UMR 7598, Laboratoire Jacques-Louis Lions, F-75005, Paris, France } \setcounter{savecntrCNRS}{\value{footnote}}\thanks{CNRS, UMR 7598, Laboratoire Jacques-Louis Lions, F-75005, Paris, France }
	\and Helge Dietert\setcounter{savecntrP7}{\value{footnote}}\thanks{Universit\'e Paris Diderot, Sorbonne Paris Cit\'e, Institut de Math\'ematiques de Jussieu-Paris Rive Gauche (UMR 7586), F-75205 Paris, France}
	\and David G\'erard-Varet\setcounter{restorecntrP7}{\value{footnote}}%
	\setcounter{footnote}{\value{savecntrP7}}\footnotemark
	\setcounter{footnote}{\value{restorecntrP7}}\; \thanks{Institut Universitaire de France, F-75205 Paris, France}
\and Fr\'ed\'eric Marbach\setcounter{restorecntrP6}{\value{footnote}}%
	\setcounter{footnote}{\value{savecntrP6}}\footnotemark
	\setcounter{footnote}{\value{restorecntrP6}}
		\setcounter{restorecntrCNRS}{\value{footnote}}%
		\setcounter{footnote}{\value{savecntrCNRS}}\footnotemark
		\setcounter{footnote}{\value{restorecntrCNRS}} 
}

\title{High frequency analysis  of the unsteady\\ Interactive Boundary Layer model}

\begin{document}
\maketitle
\begin{abstract}
  The present paper is about a famous extension of the Prandtl
  equation, the so-called {\em Interactive Boundary Layer} model
  (IBL). This model has been used intensively in the numerics of
  steady boundary layer flows, and compares favorably to the Prandtl
  one, especially past separation. We consider here the unsteady
  version of the IBL, and study its linear well-posedness, namely the
  linear stability of shear flow solutions to high frequency
  perturbations. We show that the IBL model exhibits strong
  unrealistic instabilities, that are in particular distinct from the
  Tollmien-Schlichting waves. 
We also exhibit similar instabilities for a
  \emph{Prescribed Displacement Thickness} model (PDT), which
  is one of the building blocks of numerical
  implementations of the IBL model.
\end{abstract}

\section{Introduction}

\subsection{Usual boundary layer theory}

The general concern of this paper is the boundary layer behavior of
high Reynolds number flows. We restrict to the two-dimensional case
near a flat boundary, and consider the Navier-Stokes system for
$t > 0$, $\bfx = (x,Y) \in \R \times \R_+$
\begin{equation} \label{eq:ns}
  \partial_t \bfu^\nu + \bfu^\nu \cdot \nabla \bfu^\nu + \nabla p^\nu - \nu \lap
  \bfu^\nu = 0,
  \quad \div \bfu^\nu = 0, \quad \bfu^\nu\vert_{Y=0} = 0.
\end{equation}
The parameter $\nu \ll 1$ refers to the inverse Reynolds number.

\medskip
The first description and analysis of the boundary layer goes back to the celebrated work of Prandtl, who suggested a matched asymptotic expansion for $\bfu^\nu$, of the following type
\begin{itemize}
\item away from the boundary, the Navier-Stokes solution should be described by a regular expansion in powers of $\sqrt{\nu}$, namely
\begin{equation}
\label{interior_expansion}
\bfu^\nu(t,x,Y) \sim \bfu^0(t,x,Y) + \sqrt{\nu} \bfu^1(t,x,Y) + \nu
\bfu^2(t,x,Y) + \dotsb
\end{equation}
where $\bfu^0 = (u^0,v^0)$ is the Euler solution
\begin{equation} \label{eq:euler}
\partial_t \bfu^0 + \bfu^0 \cdot \nabla \bfu^0 + \nabla p^0  = 0, \: \div \bfu^0 = 0, \: t > 0, \: \bfx \in \R \times \R_+, \quad v_0\vert_{Y=0} = 0.
\end{equation}
\item near the boundary, $\bfu^\nu = (u^\nu, v^\nu)$ should be described by an expansion of boundary layer type, with parabolic scale $\sqrt{\nu}$, namely
\begin{equation} \label{boundary_expansion}
\begin{aligned}
u^\nu(t,x,Y)  \: \sim \: & U^0(t,x,Y/\sqrt{\nu}) + \sqrt{\nu} U^1(t,x,Y/\sqrt{\nu})  + \dotsb \\
v^\nu(t,x,Y)   \: \sim \:  & \sqrt{\nu} V^0(t,x,Y/\sqrt{\nu}) + \nu V^1(t,x,Y/\sqrt{\nu})  + \dotsb
\end{aligned}\end{equation}
\end{itemize}
A similar expansion is assumed on  the pressure. After plugging expansion \eqref{boundary_expansion} into the Navier-Stokes equation \eqref{eq:ns}, one recovers the famous Prandtl  system for the leading order profile $(U^0,V^0)(t,x,y)$:
\begin{equation} \label{eq:prandtl-p0}
\begin{aligned}
& \partial_t U^0 + U^0 \partial_x U^0  + V^0 \partial_y U^0 + \partial_x P^0 - \partial^2_y U^0 =  0, \quad x \in \R, \: y > 0, \\
& \partial_y P^0 = 0, \quad x \in \R, \: y > 0, \\
& \partial_x U^0 + \partial_y V^0 = 0, \quad x \in \R, \: y > 0, \\
& U^0\vert_{y=0} =  V^0\vert_{y=0} = 0, \quad U^0 \rightarrow u_e \: \text{ as } \: y \rightarrow +\infty
\end{aligned}\end{equation}
with $u_e(t,x) := u^0(t,x,0)$.  If we assume some fast enough convergence of $U^0 - u_e$ and of its $y$-derivatives to zero as $y$ goes to infinity, system \eqref{eq:prandtl-p0} further simplifies into
\begin{equation} \label{Prandtl}
\begin{aligned}
& \partial_t U^0 + U^0 \partial_x U^0  + V^0 \partial_y U^0 - \partial^2_y U^0 = \partial_t u_e + u_e \partial_x u_e, \quad x \in \R, \: y > 0, \\
& \partial_x U^0 + \partial_y V^0 = 0, \quad x \in \R, \: y > 0, \\
& U^0\vert_{y=0} =  V^0\vert_{y=0} = 0, \quad U^0 \rightarrow u_e \quad \text{as } \: y \rightarrow +\infty
\end{aligned}\end{equation}
still with $u_e(t,x) = u^0(t,x,0)$.

\medskip
The Prandtl model \eqref{Prandtl} is the cornerstone of our
understanding of the boundary layer behavior. One striking success of
this model is the description of steady high Reynolds number flows
along a thin plate. If we model the thin plate by the half line
$\{x > 0, y = 0\}$ and follow the Prandtl theory, we end up with a
system of the type
\begin{equation} \label{Blasius}
\begin{aligned}
\partial_t U^0+ U^0 \partial_x U^0  + V^0 \partial_y U^0 - \partial^2_y U^0 = 0, \quad x \in \R_+, \: y > 0, \\
\partial_x U^0 + \partial_y V^0 = 0, \quad x \in \R_+, \: y > 0, \\
\quad U^0\vert_{y=0} =  V^0\vert_{y=0} = 0, \quad U^0 \rightarrow 1 \quad \text{as } \: y \rightarrow +\infty.
\end{aligned}\end{equation}
In the steady case, a solution of this equation is given by the  self-similar {\em Blasius flow}
\begin{equation}
 U^0(x,y) =  f\left(\frac{y}{\sqrt{x}}\right)
\end{equation}
for a profile $f$ satisfying an integrodifferential equation, see \cite{blasius1907grenzschichten}. Indeed, experiments and simulations show that this Blasius solution is  an accurate approximation of the flow for Reynolds numbers up to $10^5$.

\medskip
Still, the range of validity of the Prandtl approximation is limited, due to hydrodynamic instabilities. Two destabilizing mechanisms are well known:
\begin{itemize}
\item The so-called {\em Tollmien-Schlichting instability}. It is the classical Navier-Stokes instability of monotonic shear flows at high Reynolds numbers. This instability affects in particular the Blasius flow. We provide further comments in \cref{sec:TS}.
\item The so-called {\em separation phenomenon}. When the pressure gradient $\partial_x p^e := -\partial_t u_e - u_e \partial_x u_e$ is adverse (positive for positive $u_e$), it usually generates some recirculation (reverse flow),  with  streamlines detaching from the boundary, and a turbulent wake appearing behind the recirculation zone. We refer to \cite{guyon,DalMas,} for description of this phenomenon in the steady case, and to \cite{gargano1,gargano2} in the unsteady one.
\end{itemize}
It is well acknowledged that the Prandtl model is unable to reflect properly this kind of instabilities. In particular, separation in the steady case yields a blow up of the equations, while different types of blow-up or ill-posedness results are known in the unsteady setting.

\subsection{Alternative approaches}

One main flaw in the asymptotic model of Prandtl is that the Euler flow is given {\it a priori}, and forces the dynamics of the boundary layer through the data $u_e$. This is in contradiction with experimental observations at the onset of separation, see \cite{CebCou}, and suggests that a refined asymptotics should include some kind of interaction between the boundary layer and the inviscid equations. A hint in that direction was provided by Catherall and Mangler in \cite{catherall1966integration}:  instead of prescribing $u_e$ in the steady Prandtl system
\begin{equation}
\nonumber
\begin{aligned}
& U^0 \partial_x U^0  + V^0 \partial_y U^0 - \partial^2_y U^0 =  u_e \partial_x u_e, \quad x > 0, \: y > 0, \\
& \partial_x U^0 + \partial_y V^0 = 0, \quad x > 0, \: y > 0, \\
& U^0\vert_{y=0} =  V^0\vert_{y=0} = 0, \quad U^0 \rightarrow u_e \quad \text{as } \: y \rightarrow +\infty
\end{aligned}\end{equation}
they prescribed the {\em displacement thickness}
\begin{equation}\label{def:DT}
  \delta(t,x) := \int_{\R_+}
  \left( 1 - \frac{U^0(t,x,y)}{u_e(t,x)} \right) \dd y,
\end{equation}
and were able to solve numerically the boundary layer equations past separation, into a region of reverse flow.
This idea to consider $u_e$ as an  unknown and to couple the  inviscid and boundary layer equations was later formalized by Le Balleur, Carter, or Veldman \cite{le1990new,carter1974solutions,MR2557212}, giving birth to the so-called {\em Interactive Boundary Layer} theory (IBL) that we will now discuss.

\medskip
 In order to keep an inviscid-viscous interaction, we go back to asymptotic expansions \eqref{interior_expansion} and \eqref{boundary_expansion}. As these expansions  describe the same Navier-Stokes solution, they must coincide in the intermediate region $\sqrt{\nu} \ll Y \ll 1$. This implies a series of  matching conditions, starting with the relationships  $v^0(t,x,0) = 0$ and $\lim_{y \rightarrow +\infty} U^0 = u_e$. The next ones can be obtained from a Taylor expansion of \eqref{interior_expansion}. We have, for $Y=\sqrt{\nu} y \ll  1$:
 \begin{equation*}
\begin{aligned}
u^0(t,x,Y) & + \sqrt{\nu} u^1(t,x,Y) + \dotsb \\
= \: & u_e(t,x) + \sqrt{\nu} y \partial_Y u^0(t,x,0) +  \sqrt{\nu} u^1(t,x,0) + \OO\left(\nu + (\sqrt{\nu} y)^2\right).
\end{aligned}
\end{equation*}
If we further assume that the Euler flow is irrotational, then  $\partial_y u^0(t,x,0)= 0$ and we end up with the condition
\begin{equation}
u^0(t,x,Y) + \sqrt{\nu} u^1(t,x,Y) + \dots = u_e(t,x) +  \sqrt{\nu} u^1(t,x,0) + \OO\left(\nu + (\sqrt{\nu} y)^2\right).
\end{equation}
We deduce from this relation that
\begin{equation} \label{limU1}
 \lim_{y \rightarrow +\infty}  U^1(t,x,y) = u^1(t,x,0).
 \end{equation}
As regards the vertical velocity component, we find that for  $\sqrt{\nu} y \ll 1$:
\begin{equation*}
\begin{aligned}
v^0(t,x,Y) & + \sqrt{\nu} v^1(t,x,Y) + \dotsb \\
= \: & v^0(t,x,0) + \sqrt{\nu} y \partial_Y v^0(t,x,0) +   \sqrt{\nu} v^1(t,x,0) + \OO\left(\nu + (\sqrt{\nu} y)^2\right) \\
=  \: & - \sqrt{\nu} y \partial_x u_e(t,x)   +   \sqrt{\nu} v^1(t,x,0) + \OO\left(\nu + (\sqrt{\nu} y)^2\right)
\end{aligned}
\end{equation*}
On the other hand, for $\sqrt{\nu} \ll \sqrt{\nu} y \ll 1$,
\begin{equation*}
\begin{aligned}
& \sqrt{\nu} V^0(t,x,y) + \nu V^1(t,x,y)  + \dotsb \\
= \: & - \sqrt{\nu} \int_0^y \partial_x U^0(t,x,y') \dd y' - \nu  \int_0^y \partial_x U^1(t,x,y') \dd y' + \dotsb \\
= \: &  - \sqrt{\nu} \int_0^y  \partial_x U^0(t,x,y') \dd y'  + \OO(\nu y) \\
= \: & - \sqrt{\nu} \int_0^y  \partial_x u_e(t,x) \dd y' + \sqrt{\nu}  \int_0^y \left( \partial_x u_e(t,x)  -  \partial_x U^0(t,x,y') \right) \dd y'  + \OO(\nu y) \\
= \: &   - \sqrt{\nu} y \partial_x u_e(t,x) + \sqrt{\nu} \partial_x (u_e \delta)(t,x) + \OO(\nu y) + \OO( \sqrt{\nu} y^{-1})
\end{aligned}
\end{equation*}
where we have assumed for simplicity that  $\partial_x U^0(t,x,y) - \partial_x u_e(t,x) = \OO(y^{-2})$ at infinity. Recalling the definition \eqref{def:DT} of the displacement thickness $\delta$, it follows that
\begin{equation}
v^1(t,x,0) = \partial_x (u_e \delta)(t,x).
\end{equation}
In the Prandtl approach, this boundary condition allows to compute the next order term $\bfu^1$ in expansion \eqref{interior_expansion}, which in turn allows for the derivation of  the boundary layer correction $(U^1, V^1)$, notably imposing \eqref{limU1}. This procedure can be applied recursively to determine all terms in the expansion.

\medskip
{\it A contrario}, in the IBL approach, one does not derive the zero and first order term successively, but includes the first order correction in the zero order system. More precisely, the idea is that the approximations
\begin{equation*}
  \begin{aligned}
    & \bfu(t,x,Y) \sim \bfu^0(t,x,Y) + \sqrt{\nu} \bfu^1(t,x,Y) + \dotsb , \\
    & U(t,x,y) \sim U^0(t,x,y) + \sqrt{\nu} U^1(t,x,y) + \dotsb ,  \\
    & V(t,x,y) \sim V^0(t,x,y) + \sqrt{\nu} V^1(t,x,y) + \dotsb
  \end{aligned}
\end{equation*}
satisfy the boundary conditions
\begin{equation*}
 \lim_{y \rightarrow +\infty} U(t,x,y) \approx u_e(t,x), \quad
 v(t,x,0) \approx \sqrt{\nu} \partial_x (u_e \delta)(t,x),
\end{equation*}
with
$u_e(t,x) := u(t,x,0), \: \delta := \int_{\R_+} \left(1 -
  \frac{U(t,x,y)}{u_e(t,x)}\right) \dd y$. Note that in this definition of $u_e$, we have kept all terms up to the order $\sqrt{\nu}$, and not merely $u^0$. This boundary condition
will allow to keep a coupling between the inviscid (Euler) system and
the boundary layer system. Concretely, the IBL system consists of the
following set of equations:
\begin{itemize}
\item The Euler system in $\{x \in \R, Y > 0\}$, with unknown $\bfu = (u,v)$
\begin{equation} \label{Euler}
\partial_t \bfu + \bfu \cdot \nabla \bfu + \nabla p = 0, \quad \div \bfu = 0.
\end{equation}
\item The boundary layer system  in $\{x \in \R, y > 0\}$, with unknown $(U,V)$
\begin{equation} \label{IBL}
\partial_t U + U \partial_x U + V \partial_y U - \partial^2_y U  = \partial_t u_e + u_e \partial_x u_e, \quad \partial_x U + \partial_y V = 0
\end{equation}
\item The coupling boundary conditions
\begin{equation} \label{coupling}
\begin{aligned}
& U\vert_{y=0} = V\vert_{y=0} = 0, \quad \lim_{y \rightarrow +\infty} U = u_e, \quad  \text{ with } \:  u_e(t,x) := u(t,x,0), \\
& v\vert_{Y=0} = \sqrt{\nu}    \partial_x (u_e \delta), \quad \text{
  with } \:  \delta = \int_{\R_+} \left(1 -
  \frac{U(t,x,y)}{u_e(t,x)}\right) \dd y.
\end{aligned}
\end{equation}
\end{itemize}
The main novelty of this model is the second line in equation (\ref{coupling}), which
creates a strong coupling between \eqref{Euler} and \eqref{IBL}. If
the $\OO(\sqrt{\nu})$ term is neglected at the right-hand side of
(\ref{coupling}b), the equations \eqref{Euler} and \eqref{IBL} are
independent and one recovers the usual system. On the contrary, by
retaining this ``viscous'' term, we hope to keep relevant singular
perturbation effects from the Navier-Stokes equations.

\medskip
As mentioned earlier, system
\eqref{Euler}-\eqref{IBL}-\eqref{coupling} has revealed satisfactory
for the computation of steady boundary layer flows past separation. We
refer to the monographs \cite{CebCou,CouMau,Sob}, as well as to the lectures \cite{Lag} for an
extensive discussion, with the description of various numerical
schemes.  Our focus in this paper will be about the well-posedness of
the unsteady IBL system, as described by
\eqref{Euler}-\eqref{IBL}-\eqref{coupling}. We wish to investigate if
the IBL system has better well-posedness properties than Prandtl, as
suggested by the steady case. We will notably pay attention to the
linearization of the IBL system around shear flow solutions:
$U = U(y), V = 0$. In the Prandtl case, it is known that such
linearization is well-posed in Sobolev spaces when the shear flow is
monotonic (see \cite{MasWon,AWXY} for the general nonlinear case). On the
contrary, when the shear flow admits some non-degenerate critical
point, it exhibits high frequency instabilities that forbid
well-posedness outside the Gevrey setting \cite{MR2601044}. It is then natural
to see if such instabilities disappear in the IBL case. Let us mention
that this phenomenon was observed in the context of MHD boundary
layers, where a tangential magnetic field restores stability to high
frequencies and therefore Sobolev well-posedness \cite{GerPre,Tong1,Tong2}.

Still, as we will show hereafter, the well-posedness properties of the
unsteady IBL system are limited, as the system allows for various
kinds of high frequency instabilities. We detail our results in the
next section.

\section{Instability results}
\label{sec:main-results}
We will present here various instability theorems, and comment on their practical physical relevance. 
\subsection{Linearization around shear flows}

We consider the linearization of system \eqref{Euler}-\eqref{IBL}-\eqref{coupling} around a shear flow. Namely, the reference velocity field is given by
\begin{equation}
 \bfu = (1,0), \quad U = U_s(y),  \quad V = 0, \quad \text{ with } \: U_s(0) = 0, \quad \lim_{y \rightarrow +\infty} U_s(y) = 1.
\end{equation}
We assume that $U_s$ is smooth, monotone, with decaying derivatives at infinity. We refer to \cref{sec:us-assumptions} for a more precise description of our assumptions.
Let us remark that this reference flow satisfies \eqref{Euler} and \eqref{CC}, but not the boundary layer equation \eqref{IBL}: to obtain a solution of the IBL, one should rather consider some time dependent shear flow $(U(t,y), 0)$ satisfying the heat equation $\partial_t U - \partial^2_y U = 0$. Still, as we will investigate high-frequency/short-time instabilities, we expect only the initial data $U_s(y) = U(0,y)$ to play a role, and neglect the time dependence. Such a shortcut is classical in the framework of high frequency instabilities, but of course the irrelevance of the time dependence of $U$ would need to be proved rigorously, following  for instance the strategy in \cite{MR2601044}.  However this goes beyond the scope of this paper, whose goal is rather to prove the existence of exponentially growing modes for shear flows meeting certain criteria. 

\medskip
Linearizing the Euler equation \eqref{eq:euler}, we find that the perturbation (still denoted by $\bfu = (u,v)$) satisfies
\begin{equation}
 \partial_t u + \partial_x u + \partial_x p = 0, \quad \partial_t v + \partial_x v + \partial_Y p = 0, \quad \partial_x u + \partial_Y v = 0. 
\end{equation}
The linearized Prandtl system reads
\begin{equation} \partial_t U + U_s \partial_x U + V U'_s - \partial^2_y U  = \partial_t u_e + \partial_x u_e, \quad \partial_x U + \partial_y V  = 0, 
\end{equation}
while boundary and coupling conditions become
\begin{equation*} \label{CC}
\begin{aligned}
& U\vert_{y=0} = V\vert_{y=0} = 0, \quad \lim_{y \rightarrow +\infty} U = u_e, \quad  \text{ with } \:  u_e(t,x) := u(t,x,0), \\
& v\vert_{Y=0} = \sqrt{\nu}   \partial_x \int_0^{+\infty} (u_e - U)\, \dd y.
\end{aligned}
\end{equation*}
Here and throughout the paper, we consider perturbations $u,v, U,V$
that are $2\pi$-periodic in the tangential variable:
$x \in \T := \R/(2\pi\Z)$. Moreover, there is no loss of generality in
restricting to mean-free perturbations:
$\int_\T (u,v) \dd x = \int_\T (U,V) \dd x = 0$.  Eventually, we impose
the Euler part $(u,v)$ of the perturbation to be curl-free, which is
consistent with the derivation of the IBL, especially condition
\eqref{limU1}. This condition is of course preserved by the
evolution. Under this irrotationality assumption, we can further
simplify the system: we can write $v = \partial_x \psi$, and
$u_e(t,x) = -\partial_Y \psi_{|Y=0} $ for a {\em harmonic} stream
function $\psi$. It is then natural to introduce the
Dirichlet-to-Neumann operator\footnote{We use the convention $\mathcal{DN} f= -\partial_Y \psi_{|Y=0}$, where $\psi$ is the solution of $\Delta \psi=0$ in $\T\times (0, +\infty)$, $\psi_{|Y=0}=f$ for any $f\in H^{1/2}(\T)$ with zero mean value.} $\mathcal{DN}$ over zero-mean functions
of variable $x$. We remind that it is defined in Fourier by
\begin{equation} 
\widehat{\mathcal{DN} f}(k) = |k| \hat{f}(k), \quad \forall \, k \in \Z.
\end{equation}
Then, one has $\partial_x u_e = \mathcal{DN} v\vert_{Y=0}$, and the
the coupling condition leads to the identity
\begin{equation}
\partial_x u_e   =  \sqrt{\nu}  \mathcal{DN}  \partial_x
\int_0^{+\infty} (u_e - U)\, \dd y.
\end{equation}
As all terms are mean-free, we can remove the
$x$-derivative. Eventually, we write the linearized system as the
following system, with unknowns $U,V,u_e$:
\begin{equation} \label{LIBL}
  \begin{aligned}
    &\partial_t U + U_s \partial_x U + V U'_s - \partial^2_y U  = \partial_t u_e + \partial_x u_e, \quad \partial_x U + \partial_y V  = 0,\\
    & U\vert_{y=0} = V\vert_{y=0} = 0, \quad \lim_{y \rightarrow +\infty}
    U = u_e, \quad u_e = \sqrt{\nu}  \mathcal{DN}
    \int_0^{+\infty} (u_e - U)\, \dd y.
  \end{aligned}
\end{equation}
In complement to the analysis of model \eqref{LIBL}, we also consider the boundary layer system with prescribed displacement thickness (abridged as PDT in the rest of the paper). The unsteady nonlinear version of this system reads
\begin{equation} \label{DT}
\begin{aligned}
& \partial_t U + U \partial_x U + V \partial_y U - \partial^2_y U  = \partial_t u_e + u_e \partial_x u_e, \quad \partial_x U + \partial_y V = 0, \\
& U\vert_{y=0} = V\vert_{y=0} = 0, \quad \lim_{y \rightarrow +\infty} U = u_e, \\
&  \delta = \delta_0 \: \text{ given, with } \:  \delta(t,x) \: :=  \: \int_0^\infty \left(1 - \frac{U(t,x,y)}{u_e(t,x)}\right)\, \dd y.
\end{aligned}
\end{equation}
This system is important in at least two ways. First, it was used in
the pioneering work of Catherall and Mangler \cite{catherall1966integration}, and served as a starting
point of the IBL theory. More crucially, it is used in the simulation
of the IBL, following the numerical approach of Le Balleur or Carter,
see \cite{LeBalleur1984, carter1974solutions}. After linearization
around the shear flow, we find
\begin{equation} \label{LDT}
\begin{aligned}
&\partial_t U + U_s \partial_x U + V U'_s - \partial^2_y U  = \partial_t u_e + \partial_x u_e,  \quad \partial_x U + \partial_y V  = 0,\\
& U\vert_{y=0} = V\vert_{y=0} = 0, \quad \lim_{y \rightarrow +\infty}
U = u_e, \quad \int_0^{+\infty} (u_e - U)\, \dd y = \ds u_e
\end{aligned}
\end{equation}
with $\ds$ the displacement thickness associated to the shear
flow. The only difference between systems \eqref{LIBL} and
\eqref{LDT} is the relation between $u_e$ and
$\int_0^\infty (u_e - U)\, \dd y$.

\subsection{High tangential frequencies analysis}

In the linearized IBL and PDT systems, the evolution of the Fourier modes in the
tangential direction $x$ is decoupled. Therefore, we intend to prove the 
instability of these systems by exhibiting solutions with the following structure:
\begin{equation}
  \label{eq:eigenmode-time-x-form}
  \left( U(t,x,y), V(t,x,y), u_e(t,x) \right) 
  =
  \ee^{\lambda_k t} \, \ee^{\ii k x}
  \left( \modeU(y), \modeV(y), 1 \right),
\end{equation}
where $k$ characterizes the tangential frequency, $\Re \lambda_k$ the 
growth rate in time. Note that we impose the normalization condition
$\mathcal{U}(+\infty) = 1$. This excludes the special case where $u_e = 0$ in \eqref{LIBL} or \eqref{LDT}. This is not restrictive: in the case $u_e = 0$, an eigenmode would also be an eigenmode of the classical Prandtl equation, which has been studied extensively. 

Since we consider mean-free perturbations
and since the initial physical systems are real-valued, we can restrict
to positive frequencies without loss of generality.
We will use the following definition.

\begin{definition}
 Let $k \in \N^*$ and $\lambda \in \C$ satisfying $\Re \lambda > 0$.
 We say that a linear system displays a $(\lambda, k)$ instability 
 when there exists $\modeU, \modeV \in C^\infty(\R_+,\R)$ such 
 that~\eqref{eq:eigenmode-time-x-form} defines a smooth solution to this system.
\end{definition}

\subsection{Statement of the main results}
\label{sec:statements}

Here and in the sequel, we implicitly assume that the shear flow $\Us$ satisfies
 $\Us(0) = 0$ and $\Us(+\infty) = 1$ and that it is smooth, monotone, with decaying derivatives at infinity. We refer to \cref{sec:us-assumptions} for a more precise description of our assumptions.

\begin{thm}[Sufficient instability condition for PDT]
 \label{thm:pdt-sufficient}
 Let $\Us$ be a shear flow satisfying \cref{crit:pdt-sufficient}, page \pageref{crit:pdt-sufficient}.
 There exist $K,\eta > 0$ such that, for $k \geqslant K$,
 there exists $\lambda_k \in \C$ with $\Re \lambda_k \geqslant \eta k$
 such that the PDT system~\eqref{LDT} displays $(\lambda_k, k)$ instabilities.
\end{thm}

\bigskip

\begin{thm}[Sufficient instability conditions for IBL]
 \label{thm:ibl-sufficient}
 Let $\Us$ be a shear flow.
 \begin{itemize}
  \item 
  If $\Us$ satisfies \cref{crit:ibl-global-sufficient}, page \pageref{crit:ibl-global-sufficient}, there exist
  $K, \eta, \gamma_+ > 0$ such that, for $k \geqslant K$
  and $\nu > 0$ satisfying $(\sqrt{\nu} k)^{-1} \leqslant \gamma_+$,
  there exists $\lambda_{k,\nu} \in \C$ with $\Re \lambda_{k,\nu} \geqslant \eta k$
  such that the IBL system~\eqref{LIBL} displays 
  $(\lambda_{k,\nu}, k)$ instabilities.
  \item 
  If $\Us$ satisfies \cref{crit:ibl-local-sufficient}, page \pageref{crit:ibl-local-sufficient}, there exist 
  $K, \eta, \gamma_-, \gamma_+ > 0$ such that, for $k \geqslant K$
  and $\nu > 0$ satisfying $\gamma_- \leqslant (\sqrt{\nu} k)^{-1} \leqslant \gamma_+$,
  there exists $\lambda_{k,\nu} \in \C$ with $\Re \lambda_{k,\nu} \geqslant \eta k$
  such that the IBL system~\eqref{LIBL} displays 
  $(\lambda_{k,\nu}, k)$ instabilities.
  \item
 If  $\Us$ satisfies $\Us''(0) > 0$, there exist
 $\eta, \gamma_- > 0$ such that, for $(\sqrt{\nu} k)^{-1} \geqslant \gamma_-$
 and $\nu k^3 \geqslant 1$, there exists  $\lambda_{k,\nu}$
 such that the IBL system \eqref{LIBL} displays $(\lambda_{k,\nu},k)$
 instabilities with $\Re \lambda_{k,\nu} \geqslant \eta \nu k^3$.
 \end{itemize}
\end{thm}

\begin{thm}[Strong unconditional instability for IBL]
 \label{thm:ibl-strong}
 Let $\Us$ be an arbitrary monotone shear flow.
 There exist positive constants $\bar{\nu},K,S,\eta$
  such that for $\nu \leqslant \bar{\nu}$, $|k| \geqslant K$ 
  and $k \nu^{3/4} \geqslant S$, 
  there exists $\lambda_{k,\nu}$ such that the IBL system~\eqref{LIBL} displays
  $(\lambda_{k,\nu}, k)$ instabilities with
  $\Re \lambda_{k,\nu} \geqslant \eta \nu^{3/4} k^2$.
\end{thm}

\begin{thm}[Construction of unstable shear flows]
 \label{thm:construction}
 Both for the IBL and PDT systems, we can build shear flows displaying instabilities
 with chosen behavior within a given spectral domain. More precisely,
 \begin{itemize}
 
 \item PDT system:  let $\mu \in \Gamma_1$, defined in~\eqref{eq:def-gamma1}. 
 There exists a shear flow $\Us$
 and $K > 0$ such that, for $k \geqslant K$,
 there exists a sequence $\lambda_k$ with $\lambda_k \sim -\ii k \mu$ such that 
 the PDT system~\eqref{LDT} displays $(\lambda_k, k)$ instabilities.
 
 \item IBL system: let $\mu \in \Gamma_2$, defined in \eqref{eq:def-gamma2}. 
 For any $\epsilon > 0$, there exists a shear flow $\Us$, $K > 0$ and $\gamma_+ > 0$ such that for, $k \geqslant K$ and $(\sqrt{\nu} k)^{-1} \leqslant \gamma_+$, there
 exists a sequence $\lambda_{\nu,k}$ with $|\lambda_{\nu,k}/k + \ii \mu|
 \leqslant \epsilon$ such that the IBL system \eqref{LIBL}
 displays $(\lambda_{k,\nu},k)$ instabilities.
 
 \item IBL system: let $\mu \in \Gamma_3$, defined in \eqref{eq:def-gamma3}. 
 For any $\gamma > 0$, for any $\epsilon > 0$, 
 there exists a shear flow $\Us$, $K > 0$ and $\gamma_- < \gamma < \gamma_+$
 such that, for $k \geqslant K$, for $\gamma_- \leqslant (\sqrt{\nu} k)^{-1} \leqslant \gamma_+$,
 there exists a sequence $\lambda_{k,\nu}$
 with $|\lambda_{\nu,k}/k + \ii \mu|
 \leqslant \epsilon$ such that the IBL system \eqref{LIBL}
 displays $(\lambda_{k,\nu},k)$ instabilities.

 \end{itemize}
\end{thm}

\subsection{Comments on the main results}

\begin{remark}[Examples of unstable shear flows] \label{rk:examples}
Our Criteria~1, 2 3 for instability are of practical interest: they amount to checking the winding number of a closed plane curve around some points, and this can be done numerically by plotting the curve: see paragraph \ref{subsec:counting}. Moreover, we are able to give explicit examples of shear flows $\Us$ meeting Criteria~1, 2 or 3, to which \cref{thm:pdt-sufficient}
 or \cref{thm:ibl-sufficient} applies. For instance: 
 \begin{itemize}

 \item  Let $\alpha \in [3, 4)$, and define
 \begin{equation*}
  \hs(u) := - \log(1{-}u) - \alpha \frac{u^2}{2}.
 \end{equation*}
 Then $\hs$ is a monotone bijection from $[0, 1)$ to $[0, + \infty)$. Define $\Us:=\hs^{-1}$. Then $\Us$ satisfies Criteria 1 and 2 (see \cref{ssec:crit2-spec} for a proof).

 \item Let $\Us$ be a monotone shear flow satisfying the decay assumptions of \cref{sec:us-assumptions}. Assume furthermore that $\Us''(0)>0$ and that $\Us''$ vanishes at a single point on $(0, + \infty)$. Then $\Us$ satisfies \cref{crit:ibl-local-sufficient}.

\end{itemize}

Furthermore, let us insist that the instability pictured in Theorem \ref{thm:ibl-strong} is valid for any shear flow. This is related to the fact that the instability in Theorem \ref{thm:ibl-strong} is induced by viscosity, and is not present at the inviscid level (see Remark \ref{rk:famouscriteria}).

\end{remark}

\begin{remark}[Physically relevant regime] \label{rk:physical}
 The approximation of solutions of the Navier-Stokes system 
 \eqref{eq:ns} by solutions of the Prandtl system \eqref{Prandtl}
 within the vicinity of the boundary requires to assume that
 one can neglect the horizontal diffusion in front of the other
 terms. This is the case when $\sqrt{\nu} k \ll 1$. 
 The first instability of \cref{thm:ibl-sufficient} corresponds to a regime
 where $\sqrt{\nu} k \gg 1$. Hence, it is a theoretical mathematical regime which is located far away from the ``physical'' regime $\sqrt{\nu} k \ll 1$ for
 which the boundary layer equations are derived.
 The second instability of \cref{thm:ibl-sufficient} corresponds to values
 of $\sqrt{\nu} k \approx 1$ and is therefore near physical.
 The third instability of \cref{thm:ibl-sufficient} holds precisely
 in the physical regime $\sqrt{\nu} k \ll 1$. 
 Its instability rate $\eta \nu k^3 t$ can be seen as $\eta (\sqrt{\nu} k)^2 k t$.
\end{remark}

\begin{remark}[Comparison to other instability criteria] \label{rk:famouscriteria}
Our instability Criteria~1, 2 3 are based on an inviscid approximation of our boundary layer models. From this perspective, it is worth relating them to well-known  criteria for shear flow instability within the Euler equation. Following \cite[Chapter 4, Section 22]{DrazinReid}, one has notably in mind:
  \begin{description}
  \item[i)]  Rayleigh's criterion: a necessary condition for point spectrum instability is that there exists  $y_s\in \R$ with $U_s''(y_s)=0$;
  \item[ii)]  Fj{\o}rtoft's criterion: a necessary condition for point spectrum instability is that there exists $y\in \R$ such that $U_s''(y) (U_s(y)-U_s(y_s)) < 0$, where $y_s$ is an inflexion point.
  \end{description}
  
Indeed, it will be clear from the formulation of Criteria 1, 2 and 3 (see page \pageref{crit:pdt-sufficient}) that our monotonic shear flow profiles
satisfy both Rayleigh and Fj{\o}rtofts criteria. Still, one of the achievements of the present paper is that our criteria are sufficient for instabilities, while those just above are only necessary conditions. 

Let us stress that such inviscid instabilities are described partially in article \cite{tutty-cowley-1986-stability-unsteady-interactive-boundary-layer}, with consideration of the fixed displacement thickness system and the triple deck system. More precisely, the paper \cite{tutty-cowley-1986-stability-unsteady-interactive-boundary-layer} contains  adapted versions of the necessary conditions i) and ii) for inviscid instability, and numerical evidence of growing modes for some profiles. Still, the authors remark that their Rayleigh's and Fjortoft's type conditions  are not sufficient for instability. Moreover, as regards viscous instabilities, these conditions are seen in the numerics to be unnecessary, with  unstable flows without inflexion points. Note that since the full IBL model is not considered in \cite{tutty-cowley-1986-stability-unsteady-interactive-boundary-layer}, Theorems \ref{thm:ibl-sufficient} and \ref{thm:ibl-strong} are completely new, even at the numerical level. In particular, the worse instability described in the results above, namely the one from Theorem \ref{thm:ibl-strong}, has not been described previously, to the best of our knowledge.
   
  We also refer to 
  \textcite{cebeci-1983-time-dependent-steady-inverse-boundary-layer} for previous numerical works in the PDT model.
\end{remark}

\begin{remark}[Instability sources]
The instability mechanisms described in Theorems \ref{thm:pdt-sufficient} and \ref{thm:ibl-sufficient} on one hand, and in Theorem \ref{thm:ibl-strong} on the other hand are of a very different nature. In Theorems \ref{thm:pdt-sufficient} and \ref{thm:ibl-sufficient}, the instability is of an
inviscid nature, and transport terms are of paramount importance. On the contrary, in \cref{thm:ibl-strong}, 
 the viscous term $-\partial^2_y U$ plays a
crucial role, through a boundary layer phenomenon in the limit
$k \rightarrow +\infty$. The transport terms are only perturbative.
This
difference reflects into the dependence of the eigenvalue with respect
to $k$.  In the former case,
the growth rate scales linearly with $k$, like in the Kelvin-Helmoltz
instability. In the latter case, the real part grows like $k^2$, similarly
to what happens for the backward heat equation. Let us also point out that the instability depicted in Theorem \ref{thm:ibl-strong}
is valid for any shear flow, even with strictly concave $U_s$. For Theorems \ref{thm:pdt-sufficient} and \ref{thm:ibl-sufficient}, relying on Criteria 1, 2, 3 page \pageref{crit:pdt-sufficient}, a loss of concavity is needed. These results come
from a detailed analysis of the {\em inviscid} version of the IBL
system, which is reminiscent of the approach of Penrose for the
stability of homogeneous equilibria of the Vlasov equation. Still, we insist on the fact that we establish the existence of exact unstable eigenmodes for the fully viscous models. In other words, although the instability can be understood perturbatively, we can go beyond the construction of approximate unstable eigenmodes. This is not an obvious task: it will be clear from our analysis that the high frequency asymptotics leads to a singular perturbation problem, for which rigorous perturbative arguments are uneasy. 
\end{remark}

\begin{remark}[Consequences]
All these high-frequency phenomena cast some doubt on the numerical stability of the unsteady models \eqref{Euler}-\eqref{IBL}-\eqref{coupling} and \eqref{DT}. Let us further stress that the unstable eigenmodes constructed in \eqref{LIBL} do not correspond to the classical Tollmien-Schlichting modes of the Navier-Stokes equations. Those modes  correspond to  the regime $k \sim \nu^{-3/8}$, and $\Re \lambda \sim \nu^{-1/4}$ (see \cref{sec:TS} for a discussion).  From this perspective, the IBL system differs from another famous extension of the Prandtl equation, namely the triple deck system, which allows to recover Tollmien-Schlichting like modes but does not seem to suffer from unrealistic instabilities \cite{Smith79}. 
\end{remark}

\subsection{Strategy of proof}
We outline here the main steps in proving our results.
\paragraph{Reduction to an ODE problem}

Our main theorems all rely on the construction of solutions
of the form~\eqref{eq:eigenmode-time-x-form}. 
Using the incompressibility condition, we can express
these solutions as
\begin{equation}
 \big( U(t,x,y), V(t,x,y), u_e(t,x) \big)
 =
 \ee^{\lambda_k t} \ee^{\ii k x}
 \big(1 - \vv'(y), \ii k (\vv(y)-y), 1 \big),
\end{equation}
where $\vv$ is a smooth function of the normal variable. 
Plugging this expression in the Prandtl boundary layer equations imposes that $\vv$ is the solution to
\begin{equation}
 \label{eq:sol-vv-problem}
 \begin{aligned}
 & \left(\mu - \Us\right) \vv' + \Us' \vv
 - \frac{\ii}{k} \vv''' = F,\\
 & \vv\vert_{y=0} = 0, \quad \vv'\vert_{y=0} = 1,
 \quad \lim_{y \rightarrow +\infty} \vv' = 0,
 \end{aligned}
\end{equation}
where $\mu := \ii \lambda_k / k$ and
\begin{equation} \label{eq:def-f}
 F(y) := 1 - \Us(y) + y \Us'(y).
\end{equation}

In fact, although \eqref{eq:sol-vv-problem} is a linear ODE, the existence of solutions with appropriate behavior at infinity is not at all obvious, and requires {\it a priori} some condition on $\mu$ and $k$. Namely, we will derive good  {\it a priori} estimates on the solutions of \eqref{eq:sol-vv-problem} in the regime $\Im \mu>0$ and $k$ large enough. These estimates allow then to prove existence and uniqueness of solutions thanks to a Lax-Milgram type argument (see \cref{thm:homo-resolvent-bounds}).

System~\eqref{eq:sol-vv-problem} is supplemented with the
boundary conditions
\begin{equation} \label{eq:self-consistency-pdt}
 \lim_{y \to +\infty} \vv(y) = \ds
\end{equation}
for the PDT system and
\begin{equation} \label{eq:self-consistency-ibl}
 \lim_{y \to +\infty} \vv(y) = \frac{1}{\sqrt{\nu} |k|}
\end{equation}
for the IBL system. 
The spectral stability analysis amounts to determining
whether the solutions to~\eqref{eq:sol-vv-problem} 
verify the extra boundary conditions \eqref{eq:self-consistency-pdt} or \eqref{eq:self-consistency-ibl}. In view of these conditions, a crucial quantity is
\begin{equation} \label{eq:def-Phi}
\Phi(\mu, k):= \lim_{y\to \infty } \vv(y).
\end{equation}
Our general idea to show instability is to find sufficient criteria on $\Us$ so that there exists a closed curve $\mathcal C$, embedded in the upper half complex plane, 
such that the winding number of $\Phi(\mathcal C, k)$ around some $\gamma \in \R_+$ is positive ($\gamma = (\sqrt{\nu}k)^{-1}$ in the case of the IBL system or $\gamma = \Delta_s$ in the case of the PDT system). Using classical results of complex analysis (see \cite[Chapter~7]{needham1998visual}), we then deduce the existence of $\mu\in \C$ with $\Im \mu>0$ such that $\Phi(\mu, k)= (\sqrt{\nu}k)^{-1}$ (for the IBL system) or  $\Phi(\mu, k)= \Delta_s$ (for the PDT system).

\paragraph{Instabilities present at the inviscid level}

In the regime where $\mu$ is of order $1$, the considered instabilities
come from the inviscid equation. We can compute explicitly
the solutions $\vv_\inviscid$ to \eqref{eq:sol-vv-problem} when 
$|k| = \infty$, and define   $\Pinv(\mu):=\lim_{y\to \infty }\vv_\inviscid(y)$. We then consider the winding number of $\Pinv(\mathcal C)$ around $\gamma$ for some specific closed curves $\mathcal C$ embedded in the upper half plane. 
We characterize the sufficient compatibility
conditions between $\mu$ and $\Us$ so that this winding number is positive. 
This is the object
of \cref{sec:inviscid}. 

Then, we prove that these instabilities
persist at the viscous level for large enough
tangential frequencies. This is done in \cref{sec:persistence}
by proving that the solutions to \eqref{eq:sol-vv-problem}
are close to the solutions of its inviscid analogue.
We show that we can build unstable shear flows in
\cref{sec:construction}.

\paragraph{Viscosity-induced instabilities for IBL}

For the IBL model, we also study a regime where $\mu$ is of order $k$. In this regime, the instabilities are caused by
the viscous term. We first construct explicit approximate solutions of the model, for which we have neglected the transport term. We prove that these approximate solutions exhibit exponential growth. Then, we prove that these instabilities persist when the transport term is restored, which yields Theorem \ref{thm:ibl-strong}, see
\cref{sec:strong} for a proof.

\subsection{Assumptions on considered shear flows}
\label{sec:us-assumptions}

Throughout this paper, we assume that the shear flow $\Us \in C^\infty(\R_+, [0,1))$ satisfies $\Us(0) = 0$,  $\Us(+\infty) = 1$, $\Us' > 0$ and $U''_s(0) \neq 0$. We also assume that $\Us''$ has a finite number of zeros in $[0, +\infty)$. 
We introduce a positive weight function $\omega$ such that
\begin{equation} \label{hyp:rho}
\omega \geqslant 1
\quad \textrm{and} \quad
C_\omega := \max \left(
\left\| \frac{1}{\omega} \right\|_{L^1},
\left\| \frac{\omega'}{\omega} \right\|_\infty,
\left\| \frac{\ln \omega(y)}{1+y} \right\|_\infty
\right) < + \infty.
\end{equation}
This class includes both exponential weights such as $y \mapsto \exp(y)$ and polynomial weigths such as $y \mapsto 1 + y^2$. Once such a $\omega$ is fixed, we define
the weighted space $L^2(\omega)$ with the norm
\begin{equation}
\| \psi \|^2_{L^2(\omega)} :=
\int_0^\infty |\psi|^2 \omega.
\end{equation}
Although the domain is not bounded, the following inequality is easily satisfied:
\begin{lemma} \label{thm:poincare}
 For any $\psi : \R_+ \mapsto \C$ with $\psi'\in L^2(\omega)$ and $\psi(0) = 0$, one has
 \begin{equation} 
 \label{estimate:l2rho_linfty}
 \| \psi \|_{\infty} \leqslant C_\omega \| \psi' \|_{L^2(\omega)}.
 \end{equation}
\end{lemma}

In the sequel, we assume that the weight function $\omega$ is fixed
and that the considered shear flows satisfy the following decay properties:
\begin{gather}
 \label{eq:hyp-ds}
 \ds := \int_0^{+\infty} (1 - \Us)\, \dd y < \infty, \\
 \label{eq:hyp-uj-infty}
 \Us^{(j)} \in L^\infty
 \quad \textrm{for} \quad 1 \leqslant j \leqslant 3, \\
 \label{eq:hyp-uj-2}
(1-\Us)\in L^2(\omega),\quad y \mapsto (1+y) \Us^{(j)}(y) \in L^2(\omega)
 \quad \textrm{for} \quad 1 \leqslant j \leqslant 3.
\end{gather}
Furthermore, we assume the following decay ratio properties:
\begin{gather}
 \label{eq:hyp-us-ratio1}
 \lim_{y\to\infty} \frac{(1-U_s(y))^2}{U_s'(y)} = 0, \\
 \label{eq:hyp-us-ratio2}
 \lim_{y\to\infty} (1-U_s(y))\, y = 0, \\
 \label{eq:hyp-us-ratio3}
 \frac{(1-U_s)^2}{(U_s')^2}\, U_s'' \in L^1(\R_+).
\end{gather}
Finally, we also assume that there exists $0 < \kappa < \frac{1}{2}$ 
and $y_\kappa, c_\kappa > 0$ such that
\begin{equation} \label{eq:hyp-kappa-Us}
(1-\Us(y))^{2-\kappa} \Us''(y) \leqslant - c_\kappa (\Us'(y))^3
\quad \textrm{for} \quad y \geqslant y_\kappa.
\end{equation}
This property only depends on the asymptotic behavior of the
shear flow. It is for example satisfied if the convergence towards
the limit value is polynomial, exponential, or exponential of a polynomial.
It also implies that $\Us''(y)$ is negative for $y \geqslant y_\kappa$.

\section{Characterization of inviscid instabilities}
\label{sec:inviscid}
The purpose of this section is to study the inviscid problem
\begin{equation} \label{eq:sol-vvinv-problem}
 \left(\mu - \Us\right) \vv_{\inviscid}' + \Us' \ \vv_{\inviscid} = F, \quad
 \vv_{\inviscid}\vert_{y=0} = 0,
\end{equation}
where $F$ was defined in \eqref{eq:def-f}. Indeed, in the limit of high tangential frequencies, we expect the third order term in \eqref{eq:sol-vv-problem} to play a less important role and we approximate
\eqref{eq:sol-vv-problem} by \eqref{eq:sol-vvinv-problem}.
Let $\Pinv(\mu):=\lim_{y\to \infty}\vv_{\inviscid}(y)$. In order for $\vv_\inviscid$ to satisfy the matching boundary condition at infinity 
(\eqref{eq:self-consistency-pdt} in the PDT case or \eqref{eq:self-consistency-ibl} in
the IBL case), we investigate in the following subsections whether the equation
$\Pinv(\mu)=\gamma$ can have roots for some $\gamma \in \R_+$, with
$\Im \mu > 0$ so that it leads to an instability. To that end, we look  for  closed curves 
$\mathcal C \subset \{z\in \C, \Im(z)>0\}$ such that the winding number of $\Pinv(\mathcal C)$ around $\gamma$ is positive. This amounts eventually to counting the number of crossings of $\Pinv(\mathcal C)$ with the real axis, and we exhibit criteria on $\Us$ that ensure that the winding number is positive.

\subsection{Inviscid spectral problem}

Let us first point out that \eqref{eq:sol-vvinv-problem} can be solved explicitely as
\begin{equation}
\label{eq:inviscid-approximation-v}
\vv_{\inviscid}(y) = \left( \mu - \Us(y)\right)
\int_0^y \frac{F}{\left( \mu - \Us\right)^2}
\end{equation}
which provides the formula
\begin{equation} \label{eq:phi_inv}
 \Pinv(\mu) = \lim_{y \to +\infty} \vv_{\inviscid}(y)
 = \left( \mu - 1\right) \int_0^\infty \frac{F}{\left( \mu - \Us\right)^2}.
\end{equation}
 We start with a few technicals results, whose proof are postponed to the appendix. First, the function $\Pinv$ is well-behaved outside 
of the range of $\Us$.  Moreover, potential roots of $\Pinv(\mu) = \gamma$ for $\gamma \in \R_+$ 
and $\Im \mu > 0$ cannot be located anywhere in the complex plane. They must
be located within the disc of radius $\frac{1}{2}$ and centered at $\frac{1}{2}$.

\begin{lemma} 
 \label{thm:possible-mu}
 The map $\mu \to \Pinv(\mu)$ is holomorphic on $\C \setminus [0,1]$.
 Furthermore, for $\mu = a + \ii b$ with $b > 0$ and $b^2 \geqslant a(1-a)$, there holds $\Im \Pinv(\mu) < 0$.
\end{lemma}

We now turn towards the behavior of $\Pinv$ on $[0,1]$. Although we could fear that $\Pinv$ has a singular behavior in the vicinity
of $[0,1]$, its turns out that we can compute, for any given abscissa 
 $a \in (0,1)$, the limit of $\Pinv(\mu)$ as $\mu \to a$. We obtain this
limit using the so-called Plemelj formula, of which we use the 
following quantitative version, uniform on any subinterval.

\begin{lemma}[Plemelj formula] \label{thm:plemelj}
 Let $0 < a_0 < a_1 < 1$. There exists a constant $C > 0$ such that, 
 for $a \in [a_0, a_1]$ and $b > 0$ small enough, one has
 \begin{equation} \label{eq:plemelj}
  \left| \Pinv(a+\ii b) - G(a) \right| \leqslant C \sqrt{b},
 \end{equation}
 where we define, for $a \in (0,1)$,
 \begin{align}
  \label{eq:def-g}
  g(a) & := \frac{(1-a)^2}{\Us'(\Us^{-1}(a))^3} \Us''(\Us^{-1}(a)), \\
  G(a) & := \frac{1}{a}\frac{1}{\Us'(0)} 
  + \PV \int_0^1 \frac{g(u)}{u-a} \dd u
  + \ii \pi g(a),
 \end{align}
 where $\PV$ denotes the usual principal value operator.
\end{lemma}

\begin{remark}\label{rem:crossing}
 For $a \in (0,1)$,
\begin{equation}
\Im (G(a))=0 \iff \Us''(\Us^{-1}(a))=0.
\end{equation}
This observation will be crucial in counting and classifying the crossings of
the curve $G((0,1))$ with the real axis.
\end{remark}

\cref{thm:plemelj} allows to get a good description of $\Pinv(\mu)$ when $\mu$
is not too close to the endpoints $0$ and $1$ of the singular segment $[0,1]$.
Luckily enough, we will not need such a precise description of $\Pinv(\mu)$ near 
these points. Indeed, the result below states that: for $\mu \approx 1$, $\Pinv(\mu)$ is below
the real axis and, for $\mu \approx 0$, 
$\Pinv(\mu)$ either has a non-small imaginary part or a large positive real part.

\begin{lemma}[Behavior of $\Pinv$ near the end points of  {$[0,1]$}]
 \label{thm:phi-inv} 
 There exist positive constants $c, \rho > 0$ such that one has
 the following behaviors.
 \begin{itemize}

 \item Near 0: 
 for $a \in [0,\rho]$ and $b \in (0,\rho]$,
 one has either $|\Im \Pinv(a+\ii b)| \geqslant c$ 
 or $\Re \Pinv(a+\ii b) \geqslant c / \sqrt{b}$.
 Moreover, if $b^2 \geqslant a(1-a)$, then
 $\Im \Pinv(a +\ii b)\leqslant -c$.
 
 \item Near 1: 
 for $a \in [1-\rho,1]$ and $b \in (0,\rho]$, 
 one has $\Im \Pinv(a + \ii b) \leqslant - c b^\kappa$, where
 $\kappa$ corresponds to the exponent in assumption \eqref{eq:hyp-kappa-Us}.

 \end{itemize}
\end{lemma}
We also have a similar statement concerning the function $G$ itself:
\begin{lemma}[Behavior of $G$ near 0] \label{thm:properties-G}
 There exist $c, \rho > 0$ such that
 $\Re G(a) \geqslant c / a$ for $a \in (0,\rho]$.
\end{lemma}

\subsection{Statement of instability criteria}
\label{sec:criteria}

We have now gathered enough material to state our criteria, 
 that turn out to be sufficient conditions for the existence of
instabilities at the inviscid level. For a monotone shear flow $\Us$, 
we define $\hs := \Us^{-1}$. We consider the curve $G((0,1))$, and we denote by $y_1, \dotsc, y_k$ the zeros of $\Us''$. As explained in Remark \ref{rem:crossing}, each zero of $\Us''$ corresponds to a crossing between the curve and the real axis: more precisely, $\Im G(\Us(y_i))=0$ for $1 \leqslant i \leqslant k$. The corresponding crossing abscissa are defined by
\begin{equation} \label{eq:chi}
 \chi(y) := \Re(G(\Us(y)))=
 \frac{\hs'(0)}{\Us(y)} - \PV \int_0^1 \frac{(1-u)^2 \hs''(u)}{u-\Us(y)} \dd u.
\end{equation}
We consider different sub-groups of the $y_i$'s, depending on the nature of the crossing:
\begin{equation}
 \begin{aligned}
  I(\gamma) & :=\{i\in \{1,\dotsc, k\},\ \chi(y_i)=\gamma\}, \\
  I_+(\gamma) & :=\{i\in \{1,\dotsc, k\},\ \chi(y_i)<\gamma \text{ and }\Us''(y) >0\text{ (resp. $<0$) for }y \text{ in a }\\
  & \qquad \text{neighborhood on the right (resp. on the left) of } y_i\},\\
  I_-(\gamma) & :=\{i\in \{1,\dotsc, k\},\ \chi(y_i)<\gamma \text{ and }\Us''(y) <0\text{ (resp. $>0$) for }y \text{ in a } \\ 
  & \qquad \text{neighborhood on the right (resp. on the left) of } y_i\}.
 \end{aligned}
\end{equation}
We define
\begin{equation}
 \xi_+(\gamma)  := \mathrm{Card~} I_+(\gamma), \quad
 \xi_-(\gamma)  := \mathrm{Card~} I_-(\gamma), \quad \xi(\gamma)=\mathrm{Card~}I(\gamma).
\end{equation}
Now we can state the main sufficient conditions for the existence of 
instabilities at the inviscid level that will persist at the viscous level.

\begin{criterion} \label{crit:pdt-sufficient}
 Assume that $\xi(\ds) = 0$ and $\xi_-(\ds) > \xi_+(\ds)$.
\end{criterion}

\begin{criterion} \label{crit:ibl-global-sufficient}
 Assume that $\xi(0) = 0$ and $\xi_-(0) > \xi_+(0)$.
\end{criterion}

\begin{criterion} \label{crit:ibl-local-sufficient}
 Assume that there exists $\gamma > 0$ such that $\xi(\gamma) = 0$ and $\xi_-(\gamma) > \xi_+(\gamma)$.
\end{criterion}

\subsection{Counting roots with positive imaginary part} \label{subsec:counting}

In order to determine whether there exists a $\mu \in \C$ with 
$\Im \mu > 0$ such that $\Pinv(\mu)$ achieves a given positive real value,
we use the argument principle (see e.g. \cite[Chapter~7]{needham1998visual}). Moreover, we know
that such possible roots must lie within the region $a \in (0,1)$ and 
$b^2 < a(1-a)$. Last, the function $\Pinv$ is holomorphic except on the
line segment $[0,1]$. Hence, to count the possible roots with positive imaginary part, it is natural to introduce, for $0 < \eta \ll 1$, the contour $\mathcal{C}_\eta$ 
which is constructed as in \cref{fig:contour}. 
We then study the images of these curves by the map $\Pinv$
(see \cref{fig:ice} for some examples on given shear flows).
We now prove the main result of this section, which can be applied to any of the criteria given
in \cref{sec:criteria}.

\begin{figure}[!ht]
 \centering
 \includegraphics{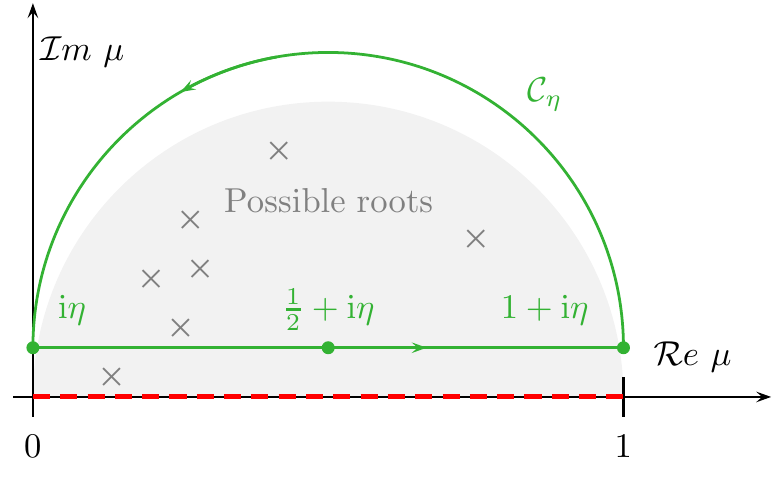}
 \caption{Closed curve used to trap possible roots with positive imaginary part}
 \label{fig:contour}
\end{figure}

\begin{figure}[!ht]
 \centering
 \includegraphics[width=13cm]{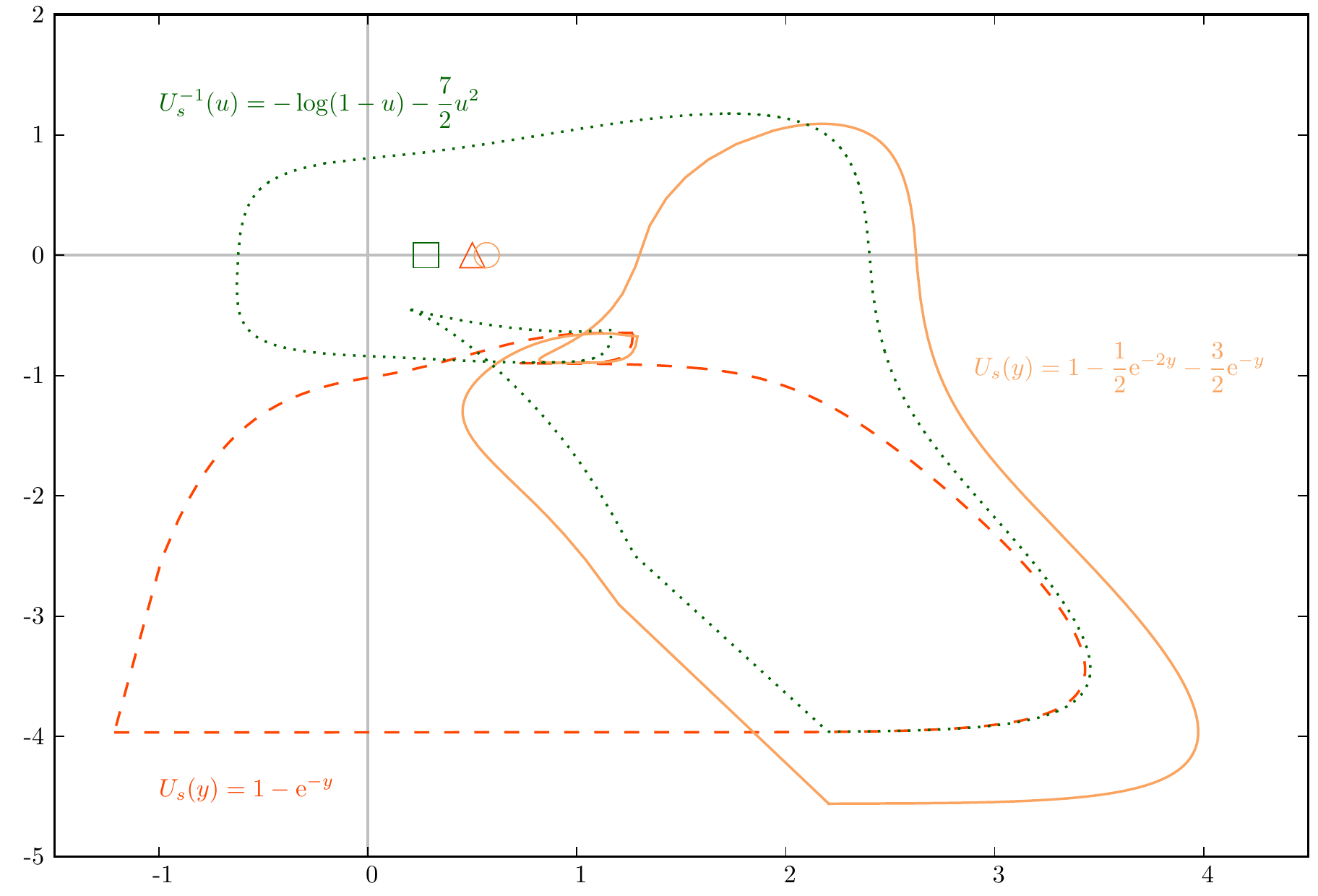}
 \caption{Illustration of curves $\Pinv(\mathcal{C}_\eta)$ for $\eta = 10^{-3}$ and three different shear flows:
 a concave shear flow $\Us(y) = 1 - \ee^{-y}$ (dashed line), that does not satisfy any of our criteria;  a profile satisfying \cref{crit:ibl-local-sufficient}, defined by
 $\Us(y) = 1 + \frac{1}{2} \ee^{-2y} - \frac{3}{2} \ee^{-y}$ (solid line); and a profile given by its inverse
 $\Us^{-1}(u) = - \log(1-u) - \frac{7}{2}u^2$ (dotted line) which satisfies all three criteria. The positions of the displacement thicknesses
 are indicated respectively by a triangle, a circle and a square. The curves are plotted in the scaled complex 
 plane by formula $\Re /(1+|\Re|^{0.8})$ and $\Im / (1+|\Im|^{0.8})$.}
 \label{fig:ice}
\end{figure}

\begin{proposition} \label{thm:winding}
 Let $\gamma \geqslant 0$.
 Assume that $\xi(\gamma) = 0$. Then for $\eta > 0$ small enough the winding number of $\Pinv(C_{\eta})$ equals $\xi_-(\gamma) - \xi_+(\gamma)$. In particular, if it is positive there exists $\mu \in \C$ with $\Im \mu > \eta$ such that $\Pinv(\mu) = \gamma$.
\end{proposition}

\begin{proof}
 \textbf{Heuristic.}
 The closed curve $\mathcal{C}_\eta$ consists of two parts: the line segment $\ii \eta + [0,1]$
 and a half-circle. For any $\mu$ in this half-circle, $\Im \Pinv(\mu) < 0$ thanks to
 \cref{thm:possible-mu}. So this part of the curve does not have any chance to wind around
 a point of the real axis. 
 The key idea of the proof is that, as $\eta \to 0$, we can use \cref{thm:plemelj} to
 estimate the behavior of the curve $\Pinv(\ii \eta + [0,1])$. However, as the ``limit
 curve'' $G((0,1))$ is unbounded and the convergence is not uniform, we must be careful.
 
 \bigskip \noindent \textbf{Controlling the behavior near the endpoints.}
 First, there exists $\rho_1 > 0$ such that $g(a)$ does not change sign
 on $[0,\rho_1]$ or on $[1-\rho_1,1)$.
 Then, from \cref{thm:phi-inv}, there exists $c, \rho > 0$ (with $\rho < \rho_1$) 
 such that,
 for $a \in [1-\rho,1)$ and $b\in(0,\rho]$, $\Im \Pinv(a+\ii b) < 0$
 and, for $a \in [0,\rho]$ and $b\in (0,\rho]$, either
 $|\Im \Pinv(a+\ii b)| \geqslant c$ or $\Re \Pinv(a+\ii b) \geqslant c/\sqrt{b}$.
 Hence, if $b \leqslant c^2 / (|\gamma|+1)^2$, then either
 $|\Im \Pinv(a+\ii b)| \geqslant c$ or 
 $\Re \Pinv(a+\ii b) \geqslant \gamma + 1$. In particular,
 the curve will not intersect the half line $(-\infty,\gamma]$ and thus
 cannot wind around $\gamma$.

 \bigskip \noindent \textbf{Construction of a reference curve.}
 We start by building a reference oriented curve $D_0$ which is composed
 of: $G([\rho,1-\rho])$ followed by a line segment between $G(1-\rho)$ and 
 the point $|\gamma|+1$ and a line segment between $|\gamma|+1$ and $G(\rho)$.
 Thanks to our choice of $\rho$, $G(1-\rho)$ is below the real axis, and $\Re(G(\rho))>|\gamma|$ (see Lemma \ref{thm:properties-G}). So the winding number
 of $D_0$ around $\gamma$ is $\xi_-(\gamma) - \xi_+(\gamma)$  (see \cref{fig:winding} for two examples).
 
\begin{figure}[!ht]
 \centering
 \includegraphics[width=\textwidth]{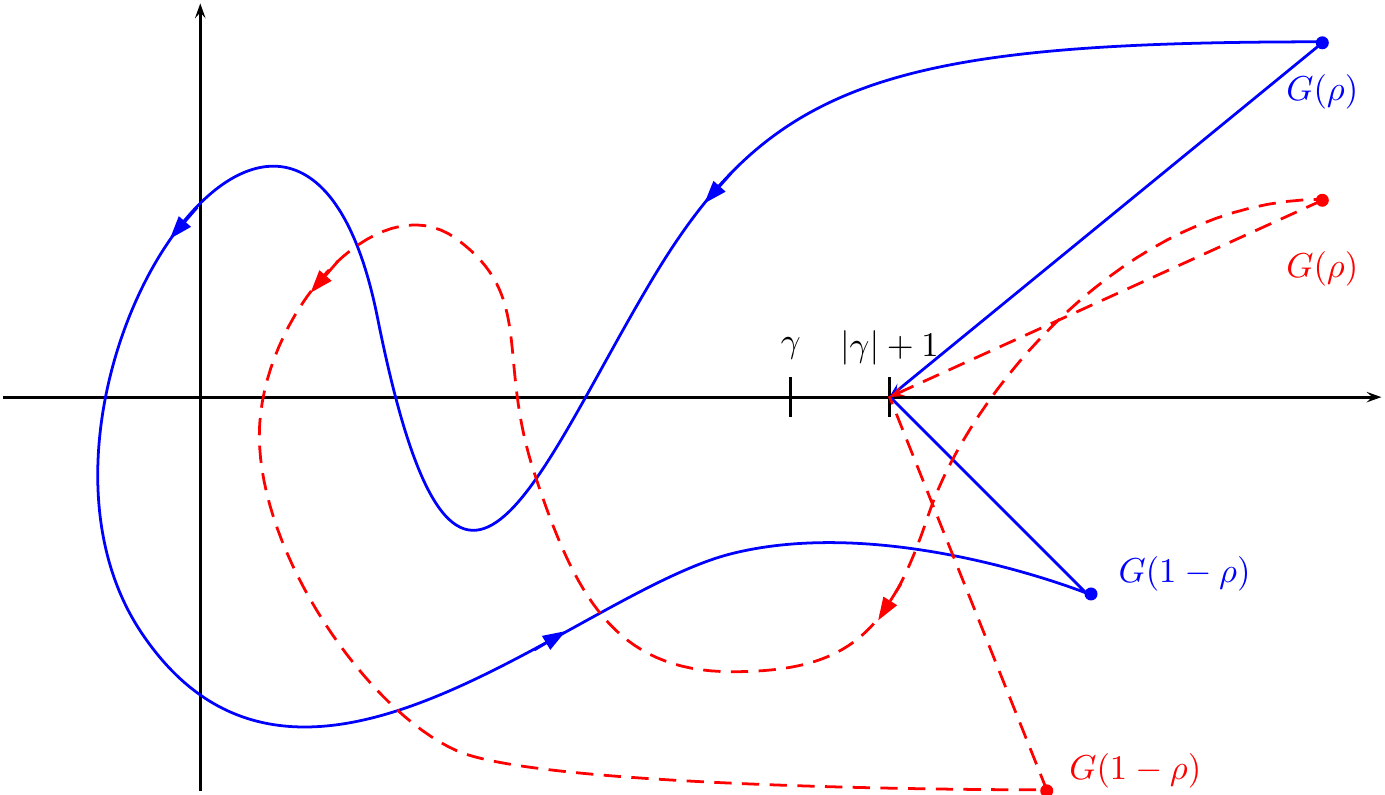}
 \caption{Two examples of reference curves $D_0$. In solid blue, $\xi_-(\gamma)=2$ and $\xi_+(\gamma)=1$, so that the winding number of the curve around $\gamma$ is equal to one. In dashed red, $\xi_-(\gamma)=\xi_+(\gamma)=1$ so that the winding number of the curve around $\gamma$ is zero.}
 \label{fig:winding}
\end{figure}
 
 \bigskip \noindent \textbf{Construction of approximate curves.}
 For $\eta \in (0,\rho)$, we define a closed oriented curve $D_\eta$
 which is composed of $\Pinv([\rho+\ii \eta, (1-\rho)+\ii\eta])$
 followed by a line segment between $\Pinv(1-\rho+\ii\eta)$
 and $|\gamma|+1$ and a line segment between $|\gamma|+1$ and 
 $\Pinv(\gamma+\ii \eta)$. The winding number of $D_\eta$
 around $\gamma$ is the same as the winding number of $\Pinv(\mathcal{C}_\eta)$
 around $\gamma$. Indeed, by Lemma \ref{thm:phi-inv}, the differences between these curves
 do not intersect the half line $(-\infty,\gamma]$
 so they cannot change the winding number. 
 
 \bigskip \noindent \textbf{Conclusion.}
 Thanks to \cref{thm:plemelj},
 the curves $D_\eta$ converge uniformly (say in the sense of
 the Hausdorff distance) to $D_0$. So, for small enough
 $\eta > 0$, the winding number of $\Pinv(\mathcal{C}_\eta)$ around $\gamma$ 
 is equal to the one of $D_0$ and therefore given by the formula of Proposition \ref{thm:winding}.
\end{proof}

\subsection{Instabilities in the physical regime}

We prove the following analogue of \cref{thm:winding} in the physical 
regime case. The main difference is that $\gamma$ is no longer a fixed quantity 
but tends towards $+\infty$. This requires some additional precautions.

For any closed oriented curve $\mathcal{D}$ and $\gamma \in \R$, we will denote
by $\mathrm{W}_\gamma[\mathcal{D}] \in \Z$ the winding number of $\mathcal{D}$
around $\gamma$ and $\mathrm{d}_\gamma[\mathcal{D}] \geqslant 0$ the distance between 
$\mathcal{D}$ and $\gamma$.

\begin{lemma}
 \label{thm:winding-physical}
 Assume that $\Us''(0) > 0$. There exists $d, \gamma_- > 0$ such that
 for $\gamma \geqslant \gamma_-$, there exists $\eta(\gamma) \geqslant d \gamma^{-2}$
 such that
 \begin{equation}
  \mathrm{W}_\gamma [ \Pinv(\mathcal{C}_{\eta(\gamma)}) ] = 1
  \quad \text{and} \quad
  \mathrm{d}_\gamma [ \Pinv(\mathcal{C}_{\eta(\gamma)}) ] \geqslant d \gamma^{-2\kappa},
 \end{equation}
 where $\kappa \in (0,\frac{1}{2})$ was introduced in assumption \eqref{eq:hyp-kappa-Us}.
\end{lemma}

\begin{proof}

We choose $\gamma_->0$ larger than $\max_{1\leq i \leq k} \chi(y_i)$, where we recall that the $y_i$'s are the zeros of $U_s''$.
Then we choose $c_G, \rho_G > 0$ such that, for $a \in [0,\rho_G]$,
$\Im G(a) = \pi g(a) \geqslant c_G$ and $\Re G(a) \geqslant c_G/a$
(see \cref{thm:properties-G}).
Let $c,\rho > 0$ be such that \cref{thm:phi-inv} applies.
We also fix $1-\rho < a_1 < 1$ such that $g(u) < 0$ for $u > a_1$.

Let $\gamma \geqslant \gamma_-$. 
We define $a_0(\gamma) := c_G/(\gamma+1)$. Note that up to choosing a larger $\gamma_-$, we can always assume that $a_0(\gamma)\leqslant \rho_G$.
We build a closed curve $D_0^\gamma$ by considering the curve arc $G([a_0(\gamma),a_1])$
followed by a line segment between $G(a_1)$ and $\Re G(a_1)-\ii$, then a line
segment between $\Re G(a_1)-\ii$ and $\gamma+1-\ii$,
and eventually a line segment between $\gamma+1-\ii$ and $G(a_0(\gamma))$.
Thanks to our construction, $\mathrm{W}_\gamma[D_0^\gamma] = 1$
and there exists $d_0 > 0$ independent of $\gamma$ such that
$\mathrm{d}_\gamma[D_0^\gamma] \geqslant d_0$. We define
\begin{equation} \label{eq:eta-gamma}
 \eta(\gamma) := (\gamma+1)^{-2} \min (c^2, c_G^2, (d_0 c_G / 4C)^2),
\end{equation}
where $C$ is the constant appearing in the right-hand side of the stronger version of the Plemelj formula \eqref{eq:plemelj-strong} in the Appendix.

We use the same closing procedure and build a closed curve $D_{\eta(\gamma)}^\gamma$ which consists 
of the arc $\Pinv([a_0(\gamma),a_1]+\ii\eta(\gamma))$ followed
by a line segment between $\Pinv(a_1+\ii\eta(\gamma))$ and $\Re G(a_1) - \ii$,
then a line segment between $\Re G(a_1)-\ii$ and $\gamma+1-\ii$
and finally a line segment between $\gamma+1-\ii$ and $\Pinv(a_0(\gamma)+\ii\eta(\gamma))$.

First, we check that $\mathrm{W}_\gamma[\Pinv(\mathcal{C}_{\eta(\gamma)})] = \mathrm{W}_\gamma[D_{\eta(\gamma)}^\gamma] = \mathrm{W}_\gamma[D_0^\gamma] = 1$.
We use the stronger version of the Plemelj formula of \cref{rk:plemelj-strong} in the Appendix.
From \eqref{eq:eta-gamma}, $\sqrt{\eta(\gamma)}/a_0(\gamma) \leqslant 1$.
Hence, using \eqref{eq:plemelj-strong}, the Hausdorff distance
between $\Pinv([a_0(\gamma),a_1]+\ii\eta(\gamma))$ and $G([a_0(\gamma),a_1])$
is lower than $2C\sqrt{\eta(\gamma)}/a_0(\gamma) \leqslant d_0/2$
thanks to \eqref{eq:eta-gamma}. By construction of the closing
procedure, the same is true between 
$D_{\eta(\gamma)}^\gamma$ and $D_0^\gamma$. This ensures
that $\mathrm{W}_\gamma[D_{\eta(\gamma)}^\gamma] = \mathrm{W}_\gamma[D_0^\gamma]$.
The first equality 
$\mathrm{W}_\gamma[\Pinv(\mathcal{C}_{\eta(\gamma)})] = \mathrm{W}_\gamma[D_{\eta(\gamma)}^\gamma]$
is obtained as in the previous cases because the closing procedure only
modifies the curve in a region where it cannot wind around $\gamma$.

Second, we check that there exists $d > 0$ such that 
$\mathrm{d}_\gamma[\Pinv(\mathcal{C}_{\eta(\gamma)})] \geqslant d\rho^{-2\kappa}$.
We decompose the initial curve $\mathcal{C}_{\eta(\gamma)}$ and split it in 5 different parts
depending on the value of $\mu = a + \ii b$ running along the oriented curve:
\begin{itemize}
 \item Points where $a \in [a_0(\gamma), a_1]$ and $b = \eta(\gamma)$. Here, as explained
 above, the curve is close to $G([a_0(\gamma),a_1])$ and thus at distance at least $d_0/2$ of
 $\gamma$.
 \item Points where $a \in [1-\rho,1]$ and $b \in (0,\rho]$. Here, $b \geqslant \eta(\gamma)$ and, from \cref{thm:phi-inv}, $\Im \Pinv(\mu) \leqslant - c b^\kappa \leqslant - c \eta(\gamma)^\kappa$.
 \item Points where $b \geqslant \rho$. Here, there exists $d_1 > 0$ (depending on $\rho$) such that
 $\Im \Pinv(\mu) \leqslant -d_1$ uniformly. Indeed, thanks to \cref{thm:possible-mu}, 
 the function $\Im \Pinv(\mu)$ is negative when $b^2 \geqslant a(1-a)$ and continuous for $b > 0$. 
 Hence, it has a negative maximum on the compact set $\{ a+\ii b \in \C, \, b \in [\rho,1], \, a \in [0,1], \, b^2 \geqslant a(1-a) \}$.
 \item Points where $a \in [0,\rho]$, $b \in (0,\rho)$ and $b^2 \geqslant a(1-a)$. Here,
 from \cref{thm:phi-inv}, $\Im\Pinv(\mu) \leqslant -c$.
 \item Points where $a \in [0,a_0(\gamma)]$ and $b = \eta(\gamma)$. Here, from \cref{thm:phi-inv},
 there either holds $|\Im\Pinv(\mu)| \geqslant c$ or $\Re \Pinv(\mu) \geqslant c/\sqrt{\eta(\gamma)} \geqslant \gamma+1$
 thanks to \eqref{eq:eta-gamma}.
\end{itemize}
Eventually, for $\gamma$ large enough, we get the existence of a constant $d > 0$ 
such that $\mathrm{d}_\gamma [ \Pinv(\mathcal{C}_{\eta(\gamma)}) ] \geqslant d \gamma^{-2\kappa}$
and $\eta(\gamma) \geqslant d\gamma^{-2}$.
\end{proof}

We refer to \cref{sec:proofs-123} for the conclusion of the proof
of \cref{thm:ibl-sufficient} in the physical regime case, where
we prove that these instabilities persist at the viscous level.
We can already note here that the instabilities will grow at least like
$\eta(\gamma) k t \sim \nu k^3 t$.

\section{Persistence of inviscid instabilities}
\label{sec:persistence}

As announced in the introduction, the instabilities identified in 
\cref{sec:inviscid} persist at the viscous level for high enough 
tangential frequencies. Two approaches are possible.

The first approach, which is quite natural in problems involving a small
parameter (here, the small parameter is $|k|^{-1}$), is to
construct approximate solutions of \eqref{eq:sol-vv-problem}, 
of arbitrarily good precision. The main order term in this approximate
solution will be the eigenmode constructed in \eqref{eq:inviscid-approximation-v}.
Since the full ODE \eqref{eq:sol-vv-problem} is of a higher degree than 
\eqref{eq:sol-vvinv-problem}, one must construct boundary 
layer correctors. Then, one must add internal correctors. 
Iterating the procedure allows to construct arbitrarily precise approximations,
which all blow up exponentially fast with respect to time since the
main order does. Then, one must prove that the exact solution of the evolution
equation with a suitable initial data remains close to this
approximate solution (and thus also exhibits exponential growth).
This allows to prove ill-posedness results for the evolution equation, say in Sobolev spaces
 (see e.g. \cite{MR2601044}). 

The approach in the present paper is different. It relies on more abstract continuity arguments, showing that there exists exact eigenmodes of the full viscous equation.  Our argument involves a new energy estimate for the boundary layer problem  \eqref{eq:sol-vv-problem}, uniform for large tangential frequencies.

At the end of this section, in \cref{sec:method2}, 
we give a quick description of the iterative construction involved 
in the first approach in the PDT case, as an illustration.

\subsection{Energy estimates for the homogeneous problem}
\label{sec:homogeneous-resolvent}

We start by proving an important energy estimate for the following
homogeneous resolvent problem, associated with system \eqref{eq:sol-vv-problem}:
\begin{equation}
\label{eq:sol-resolvent-v-problem-homogeneous} 
\begin{aligned}
& \left(\mu - U_s\right) \vv' + U_s' \vv
- \frac{\ii}{k} \vv''' = f,\\
& \vv\vert_{y=0} = 0, \quad \vv'\vert_{y=0} = 0,
\quad \lim_{y \rightarrow +\infty} \vv' = 0,
\end{aligned}
\end{equation}
with a general forcing $f$. We intend to obtain energy bounds that rely on the
inviscid part of the equation so that they can be uniform with respect to $k$
for large $|k|$.

\begin{lemma}
 \label{thm:homo-resolvent-bounds}
 Let $\omega$ be a weight function satisfying \eqref{hyp:rho} and let $\eta > 0$. There exist $K_\eta, C_\eta > 0$ such that, 
 for $(\mu, k) \in \C \times \N^*$ with $k \geqslant K_\eta$, 
 $\Im \mu \geqslant \eta$ and $f \in L^2(\omega)$, 
 system \eqref{eq:sol-resolvent-v-problem-homogeneous} 
 has a unique solution $\vv$ satisfying the estimate
 \begin{equation} \label{energy:v'.1}
 \| \vv' \|_{L^2(\omega)} \leqslant C_\eta \| f \|_{L^2(\omega)},
 \end{equation}
and if $|\Re \mu| \geqslant 4$, the additional estimate  
 \begin{equation} \label{energy:v'.2}
 \| \vv' \|_{L^2(\omega)} 
 \leqslant 
 C_{\eta} \left|\frac{1}{\Re \mu}\right| \| f \|_{L^2(\omega)}.
 \end{equation}
Moreover, the map $(\mu,k) \mapsto \vv$ is analytic with respect to $\mu$.
\end{lemma}

\begin{proof}
 We obtain these results through \emph{a priori} energy estimates. 
 Assume that $\vv$ is a solution to~\eqref{eq:sol-resolvent-v-problem-homogeneous}
 such that $\vv' \in L^2(\omega)$.  
 Since $\vv(0) = 0$ and $\Im \mu>0$, we can define $\psi$ such that
 \begin{equation} \label{def:v.from.phi}
 \vv(y) = (\mu - U_s(y)) \int_0^y \psi(z) \dd z.
 \end{equation}
 Hence, inverting~\eqref{def:v.from.phi} yields:
 \begin{equation} \label{def:phi.from.v}
 \psi = \frac{\vv'}{\mu - U_s} + \frac{U_s' \vv}{(\mu-U_s)^2}.
 \end{equation}
 We test equation~\eqref{eq:sol-resolvent-v-problem-homogeneous} against
 $-\ii \conj{\psi} \omega / (\mu - U_s)$. First, the inviscid part
 of the equation yields the contribution
 \begin{equation} \label{eq:testing.1}
 \begin{split}
 \int_0^\infty
 \left[\left(\mu - U_s\right) \vv' + U_s' \vv\right]
 \frac{-\ii \conj{\psi} \omega}{\mu - U_s} 
 & = \int_0^\infty -\ii (\mu-U_s) \psi \conj{\psi} \omega \\
 & = - \ii \mu \| \psi \|^2_{L^2(\omega)}
 + \ii \int_0^\infty U_s |\psi|^2 \omega.
 \end{split}
 \end{equation}
 Second, the viscous term produces four contributions:
 \begin{equation} \label{eq:viscous.4}
 \begin{split}
 -\frac{1}{k} 
 \int_0^\infty
 \frac{\vv''' \conj{\psi} \omega}{\mu - U_s}
 &= - \frac{1}{k} \int_0^\infty \psi''\conj{\psi} \omega \\
 &\phantom{=} + \frac{3}{k} \int_0^\infty \frac{U_s'}{\mu-U_s}
 \psi' \conj{\psi} \omega \\
 &\phantom{=} + \frac{3}{k} \int_0^\infty \frac{U_s''}{\mu-U_s}
 \psi \conj{\psi} \omega \\
 &\phantom{=} + \frac{1}{k} \int_0^\infty \frac{U_s'''}{(\mu-U_s)^2}
 \vv \conj{\psi} \omega.
 \end{split}
 \end{equation}
 From~\eqref{def:phi.from.v} and the conditions $\vv(0) = \vv'(0) = 0$, one has $\psi(0) = 0$. From \cref{thm:poincare}, $\vv \in L^\infty$.
 From~\eqref{def:phi.from.v} and the condition $U_s'\to 0$ at $+\infty$, one also
 has $\psi(+\infty) = 0$. Hence we can integrate by parts the first term
 in the right-hand side of~\eqref{eq:viscous.4}:
 \begin{equation} \label{eq:viscous.4.1}
 - \frac{1}{k} \int_0^\infty \psi'' \conj{\psi} \omega
 = \frac{1}{k} \| \psi' \|^2_{L^2(\omega)}
 + \frac{1}{k} \int_0^\infty \psi' \conj{\psi} \omega',
 \end{equation}
 where, using~\eqref{hyp:rho}, one has the bound
 \begin{equation} \label{eq:viscous.4.1.2}
 \left| \frac{1}{k} \int_0^\infty \psi' \conj{\psi} \omega' \right|
 \leqslant \frac{C_\omega}{|k|} \|\psi'\|_{L^2(\omega)} \|\psi\|_{L^2(\omega)}.
 \end{equation}
 The last three terms in~\eqref{eq:viscous.4} are bounded as follows
 \begin{align}
 \left|
 \frac{3}{k} \int_0^\infty \frac{U_s'}{\mu-U_s}
 \psi' \conj{\psi} \omega
 \right|
 &\leqslant \frac{3 \|U_s'\|_\infty}{|k|\eta}
 \| \psi \|_{L^2(\omega)} \| \psi' \|_{L^2(\omega)}  \label{eq:viscous.4.2}\\
  \label{eq:viscous.4.3}
 \left| \frac{3}{k} \int_0^\infty \frac{U_s''}{\mu-U_s}
 \psi \conj{\psi} \omega \right|
 &\leqslant \frac{3 \|U_s''\|_{\infty}}{|k| \eta} \| \psi \|^2_{L^2(\omega)}, \\
 \label{eq:viscous.4.4}
 \left|
 \frac{1}{k} \int_0^\infty \frac{U_s'''(y)}{(\mu-U_s)^2}
 \vv \conj{\psi} \omega \right|
 &\leqslant \frac{C_\omega \|U_s'''\|_{L^2(\omega)}}{|k|\eta} \| \psi \|_{L^2(\omega)}^2,
 \end{align}
 where the last inequality is obtained because, thanks to~\eqref{def:v.from.phi}
 and \cref{thm:poincare} applied to $\psi$, one has
 \begin{equation}
 \left\| \frac{\vv}{\mu-U_s} \right\|_{\infty} 
 = \left\| \int \psi \right\|_{\infty}
 \leqslant C_\omega \|\psi\|_{L^2(\omega)}.
 \end{equation}
 Finally, we find from testing the forcing
 \begin{equation} \label{eq:testing.f}
 \left| \int_0^\infty
 \frac{-\ii f \conj{\psi} \omega}{\mu- U_s}  \right|
 \leqslant \left\| \frac{1}{\mu -U_s} \right\|_{\infty}
 \| f \|_{L^2(\omega)} \| \psi \|_{L^2(\omega)}.
 \end{equation}
 
 \noindent \textbf{Proof of the first energy estimate.} 
 Estimate~\eqref{energy:v'.1} is obtained by considering the real part of 
 the testing procedure described above. Gathering~\eqref{eq:testing.1},
 \eqref{eq:viscous.4.1}, \eqref{eq:viscous.4.1.2}, 
 \eqref{eq:viscous.4.2}, \eqref{eq:viscous.4.3}, \eqref{eq:viscous.4.4}
 and \eqref{eq:testing.f} and using Young's inequality yields the existence of a constant $C_\eta$ depending only on $\omega$, $U_s$  and $\eta$ such that
 \begin{equation} \label{eq:merged.energy.1}
\frac{1}{2|k|} 
 \| \psi' \|^2_{L^2(\omega)}
  + \left( \Im \mu - \frac{C_\eta}{|k|} \right)
 \| \psi \|^2_{L^2(\omega)} 
  \leqslant \left\| \frac{1}{\mu - U_s} \right\|_{\infty}
 \| f \|_{L^2(\omega)} \| \psi \|_{L^2(\omega)}.
 \end{equation}
 Therefore, from~\eqref{eq:merged.energy.1} and the assumption $\Im \mu \geqslant \eta$, there exists $K_\eta$ large enough such that for $|k| \geqslant K_\eta$,
 \begin{equation} \label{energy.phi.1}
 \| \psi \|_{L^2(\omega)}
 \leqslant \frac{2}{\eta} \left\| \frac{1}{\mu - U_s} \right\|_{\infty}
 \| f \|_{L^2(\omega)}.
 \end{equation}
 Differentiating~\eqref{def:v.from.phi} gives
 \begin{equation} \label{def:v'.from.phi}
 \vv'(y) = - U_s'(y) \int_0^y \psi(z) \dd z
 + (\mu - U_s(y)) \psi(y)
 \end{equation}
 Combining~\eqref{def:v'.from.phi} and \cref{thm:poincare} gives
 \begin{equation} \label{eq:energy.v'.1.1}
 \| \vv' \|_{L^2(\omega)}
 \leqslant C_\omega \|U_s'\|_{L^2(\omega)} \| \psi \|_{L^2(\omega)}
 + \| \mu -U_s \|_{\infty} \| \psi \|_{L^2(\omega)}.
 \end{equation}
 By the assumption $\Im \mu \geqslant \eta$, one obtains:
 \begin{equation} \label{eq:cdelta}
 \| \mu -U_s \|_{\infty}
 \left\| \frac{1}{\mu - U_s} \right\|_{\infty}
 \leqslant \max \left(3, \frac{3 \|U_s\|_\infty}{\eta}\right).
 \end{equation}
 The combination of~\eqref{energy.phi.1},~\eqref{eq:energy.v'.1.1} and~\eqref{eq:cdelta} proves estimate~\eqref{energy:v'.1}.
 
 \bigskip
 \noindent \textbf{Proof of existence and uniqueness.} 
The well-posedness of the linear equation \eqref{eq:sol-resolvent-v-problem-homogeneous} follows from the application of a variant of the Lax-Milgram Lemma to a suitable variational formulation of the equation. First, following the proof of the energy estimates of  \cref{thm:homo-resolvent-bounds}, it is useful to consider the equation on 
\begin{equation}
\Psi:=\frac{\phi}{\mu-U_s}=\int_0^y \psi. 
\end{equation}
The estimates from \cref{thm:homo-resolvent-bounds} suggest that we look for $\Psi$ in the functional space
\begin{equation}
\cH:=\left\{ \theta\in H^2_\mathrm{loc}(\R_+, \C),\ \theta(0)=\theta'(0)=0,\ \|\theta'\|_{L^2(\omega)}+  \|\theta''\|_{L^2(\omega)}<+\infty\right\}.
\end{equation}
Note that equation \eqref{eq:sol-resolvent-v-problem-homogeneous} can be written in terms of $\Psi$ as
\begin{equation}\label{eq:varform-Psi}
-\frac{\ii}{k} \Psi^{(3)} + \frac{3\ii}{k} \frac{U_s'}{\mu-U_s} \Psi'' + \frac{3\ii}{k} \frac{U_s''}{\mu-U_s} \Psi'+ \frac{\ii}{k} \frac{U_s^{(3)}}{\mu-U_s} \Psi + (\mu-U_s)\Psi'= \frac{f}{\mu-U_s}.  
\end{equation}
Multiplying the above equation by $-\ii \bar \theta' \omega$ where $\theta$ is an arbitrary test function in $\cH$, we obtain the following variational formulation:
\begin{equation}\label{eq:resolvent-varform}
a(\Psi, \theta)=\int_0^\infty-\ii \frac{f}{\mu-U_s}\bar \theta' \omega ,
\end{equation}
where
\begin{eqnarray*}
a(\Psi, \theta)&=&\frac{1}{k}\int_0^\infty\left( \Psi''\bar \theta'' \omega + \Psi'' \bar \theta' \omega'\right)\\
&&+ \frac{1}{k}\int_0^\infty\left( \frac{3U_s'}{\mu-U_s}\Psi'' + \frac{3U_s''}{\mu-U_s}\Psi' + \frac{U_s'''}{\mu-U_s}\Psi\right) \bar \theta' \omega\\
&&-\ii\int_0^\infty (\mu-U_s)\Psi' \bar \theta'\omega.
\end{eqnarray*}
Using the assumptions on $U_s$ and $\omega$, $a$ is a continuous bilinear form on $\cH$. Hence by the Riesz-Fr\'echet representation theorem, there exists a linear continuous application $A:\cH\to \cH$ such that 
\begin{equation}
\langle Au, \theta\rangle= a(u,\theta)
\end{equation}
for all $\theta\in \cH$, where $\langle \theta_1, \theta_2\rangle=\int_0^\infty  \theta_1' \bar \theta_2' \omega$. We now need to prove that $A$ is a bijection from $\cH$ to $\cH$. We cannot apply exactly the  Lax-Milgram Lemma, since $a(\theta,  \theta)$ is not real-valued. However, we follow the proof outlined after the  Lax-Milgram Lemma in \cite{MR2759829}. Indeed, the energy estimates of the previous paragraph show that if $A\Psi=0$, then $\Psi=0$. Moreover, $A$ has a closed range since for all $\Psi \in \cH$, according to \eqref{eq:merged.energy.1}
\begin{equation}
\|\Psi\|_{\cH}^2 \leq C_\eta |a(\Psi,  \Psi)| \leq C \| A \Psi\|_{\cH} \| \Psi\|_{\cH}.
\end{equation}
Eventually, the image of $A$ is dense since for any $\Psi \in \cH$,
\begin{equation}
a(\Psi, \theta)=0\quad \forall \theta\in \cH \Rightarrow \Psi=0.
\end{equation}
Hence $A$ is a bijection, and there exists a unique $\Psi\in\cH$ such that
\begin{equation}
A\Psi= -\ii \frac{ f}{ \mu - U_s}.
\end{equation}
Hence $\Psi$ satisfies \eqref{eq:resolvent-varform} for all $\theta\in \cH$, and is a variational solution of \eqref{eq:varform-Psi}.

 \bigskip
 \noindent \textbf{Proof of analyticity.} 
 Like for a normal resolvent, it follows directly that the map $\mu \to \vv$ is 
 analytic in $\{ \mu \in \C, \enskip \Im \mu > 0 \}$. Indeed, let us fix some
 $\mu_0$ with $\Im \mu_0 > 0$. We apply the previous energy estimates with
 $\eta := (\Im \mu_0) / 2$. We fix some $k \geqslant K$ where $K$ is such that the
 energy estimate holds for this $\eta$. We consider the linear operator
 \begin{equation}
  \mathcal{L}_k :
  \left\{ 
  \begin{aligned}
   H^2(\omega) & \to L^2(\omega), \\
   g & \mapsto \Us g - \Us' \int_0^y g + \frac{\ii}{k} g'',
  \end{aligned}
  \right.
 \end{equation}
 where $H^2(\omega)$ denotes the Sobolev space of functions $h$ vanishing at zero and
 infinity and such that $h$, $h'$ and $h''$ belong to $L^2(\omega)$. Hence,
 $\mathcal{L}_k$ defines a bounded operator. The main energy estimate
 \eqref{energy:v'.1} proves that the resolvent $(\mu-\mathcal{L}_k)^{-1}$
 exists and is bounded for $\mu$ in a neighborhood of $\mu_0$. Hence, one can 
 write  $(\mu-\mathcal{L}_k)^{-1}$ as a von Neumann series:
 \begin{equation}
  (\mu-\mathcal{L}_k)^{-1} = \sum_{j=0}^{+\infty} (\mu_0-\mu)^j (\mu_0-\mathcal{L}_k)^{-j-1}.
 \end{equation}
 This proves that, for any $\eta > 0$ and for any large enough fixed $k$, the map $\mu \to \vv$
 is holomorphic on the half-plane $\Im \mu \geqslant \eta$.
 
 \bigskip
 \noindent \textbf{Proof of the second energy estimate.} 
 Estimate~\eqref{energy:v'.2} is obtained by considering the imaginary part 
 of the testing procedure described above. Recall that $\|U_s\|_\infty=1$.
 Gathering~\eqref{eq:testing.1}, \eqref{eq:viscous.4.1}, \eqref{eq:viscous.4.1.2}, 
 \eqref{eq:viscous.4.2}, \eqref{eq:viscous.4.3}, \eqref{eq:viscous.4.4}
 and \eqref{eq:testing.f} yields the existence of a constant $C > 0$ such that:
 \begin{equation} \label{eq:merged.energy.2}
 \left(\left|\Re \mu\right| - 1 - \frac{C}{|k|}\right)
 \| \psi \|_{L^2(\omega)}
 \leqslant \frac{C}{|k|} 
 \| \psi' \|_{L^2(\omega)}
 + \left\| \frac{1}{\mu - U_s} \right\|_{\infty}
 \| f \|_{L^2(\omega)}.
 \end{equation}
 Using~\eqref{eq:merged.energy.1} and~\eqref{energy.phi.1}, we obtain the control
 \begin{equation} \label{energy.phi'.1}
 \frac{1}{\sqrt{|k|}} \| \psi'\|_{L^2(\omega)}
 \leqslant 
  C_\eta \left\| \frac{1}{\mu - U_s} \right\|_{\infty}
 \| f \|_{L^2(\omega)}.
 \end{equation}
 Combining~\eqref{eq:merged.energy.2} and~\eqref{energy.phi'.1} yields, 
 for $|k| \geqslant K$ large enough:
 \begin{equation} \label{energy.phi.b.1}
 \left(\left|\Re \mu\right| - 2\right) \| \psi \|_{L^2(\omega)}
 \leqslant 
 2 \left\| \frac{1}{\mu - U_s} \right\|_{\infty} \| f \|_{L^2(\omega)},
 \end{equation}
 from which we obtain eventually for $|\Re \mu| \geqslant 4$ 
 \begin{equation} \label{energy.phi.b.2}
 \| \psi \|_{L^2(\omega)} \leqslant 
 4 \left|\frac{1}{\Re \mu}\right|  \left\| \frac{1}{\mu - U_s} \right\|_{\infty}
  \| f \|_{L^2(\omega)}.
 \end{equation}
 Relating this estimate of $\psi$ in~\eqref{energy.phi.b.2} 
 to $\vv'$ using~\eqref{eq:energy.v'.1.1} and~\eqref{eq:cdelta} as 
 for the previous estimate proves~\eqref{energy:v'.2}.
\end{proof}

\begin{cor}
 \label{thm:unit-forced-resolvent-problem}
 Let $\eta > 0$. 
 There exists $K$ such that for $k \geqslant K$ and $\Im \mu \geqslant \eta$, 
 system~\eqref{eq:sol-vv-problem} has a unique solution. 
 For any fixed $k$, the map $\mu \mapsto \vv$ is analytic.
\end{cor}

\begin{proof}
 Write $\vv = \vv_B + \vv_H$ where $\vv_B$ is a smooth decaying function
 satisfying the boundary conditions. Then $\vv_H$ satisfies a
 homogeneous problem of the form \eqref{eq:sol-resolvent-v-problem-homogeneous} 
 for some computable forcing and we can apply \cref{thm:homo-resolvent-bounds}.
\end{proof}

In particular, thanks to \cref{thm:unit-forced-resolvent-problem}, 
the definition of $\Phi(\mu,k) = \lim_{+ \infty} \vv$
announced in \eqref{eq:def-Phi} makes sense 
for all $\mu\in \C$ such that $\Im \mu \geq \eta>0$ 
and for $k\geq K_\eta$.

\subsection{Convergence of the inviscid approximation}

In order to show the convergence of $\Phi$ to $\Phi_{\inviscid}$ we
control the difference between the solution $\vv$ of
\eqref{eq:sol-vv-problem} and its approximation $\vv_{\inviscid}$ from
\eqref{eq:inviscid-approximation-v}.

\begin{lemma}
 \label{thm:convergence-approx-inv}
 Let $\eta > 0$. There exist constants $C, K > 0$ such that, 
 for any $(\mu, k) \in \C \times \N^*$ with $k \geqslant K$ 
 and $\Im \mu \geqslant \eta$, it holds that
 \begin{equation} \label{phi.approx.phi.inv}
 \left| \Phi\left(\mu,k\right)
 - \Phi_{\inviscid}\left(\mu\right) \right|
 \leqslant
 C |k|^{-\frac{1}{2}}.
 \end{equation}
\end{lemma}

\begin{proof}
 We decompose $\vv = \vv_{\inviscid} + \vc + \tvv$, where $\vc$ is a corrector intended to catch up the boundary condition $\vv'(0)=1$ (which is not satisfied by $\vv_{\inviscid}$ and corresponds to the no-slip property). We define $\vc$ as the solution to
 \begin{equation} \label{system:vc}
 \begin{aligned}
 & \mu \vc' - \frac{\ii}{k} \vc''' = 0, \\
 & 
 \vc\vert_{y=0} = 0, \quad 
 \vc'\vert_{y=0} = 1 - \frac{1}{\mu}, \quad 
 \lim_{y \rightarrow +\infty} \vc' = 0,
 \end{aligned}
 \end{equation}
 which can be solved explicitly as
 \begin{equation} \label{vc.explicit}
 \vc(y) = \left( 1 - \frac{1}{\mu}\right) \int_0^y \ee^{- \sqrt{-\ii k \mu}z} \dd z,
 \end{equation}
 where $\sqrt{- \ii k \mu}$ can be determined analytically such that 
 $\Re \sqrt{- \ii k \mu} \geqslant \sqrt{k \eta}$. Using~\eqref{system:vc},
 one gets that $\tvv$ is the solution to
 \begin{equation}
 \begin{aligned}
 & \left(\mu - U_s\right) \tvv' + U_s' \tvv
 - \frac{\ii}{k} \tvv'''
 = \tilde F, \\
 & \tvv\vert_{y=0} = 0, \quad \tvv'\vert_{y=0} = 0,
 \quad \lim_{y \rightarrow +\infty} \tvv' = 0
 \end{aligned}
 \end{equation}
 with
 \begin{equation} \label{def.fw}
\tilde F := - \frac{\ii}{k} \vv_{\inviscid}''' + U_s' \vc - U_s \vc'.
 \end{equation}
 We need to estimate the size of $\tilde F$ in $L^2(\omega)$ to apply \cref{thm:homo-resolvent-bounds}. First, from~\eqref{vc.explicit},
 \begin{equation} \label{fw.1}
 \left\| U_s' \vc \right\|_{L^2(\omega)} 
 \leqslant
 \|\vc\|_\infty \| U_s' \|_{L^2(\omega)}
 \leqslant
 \left(1+\frac{1}{\eta}\right) \frac{1}{\sqrt{\eta k}} \| U_s' \|_{L^2(\omega)}.
 \end{equation}
 Moreover, since $U_s(0) = 0$, writing $U_s$ as the primitive of $U_s'$ gives:
 \begin{equation} \label{fw.2}
 \begin{split}
 \left\| U_s \vc' \right\|_{L^2(\omega)}^2
 & = \left|1+\frac{1}{\mu}\right|^2 
 \int_0^{\infty} \left(\int_0^y U_s'(z) \dd z\right)^2 e^{-2 y \Re \sqrt{-\ii k \mu}}
 \omega(y) \dd y \\
 & \leqslant \left(1+\frac{1}{\eta}\right)^2 \| U_s' \|_\infty^2
 \int_0^{\infty} y^2 e^{-2 y \Re \sqrt{-\ii k \mu}} \omega(y) \dd y \\
 & \leqslant 2 C_\omega \left(1+\frac{1}{\eta}\right)^2 \| U_s' \|_\infty^2 \left(\Re \sqrt{- \ii k \mu}\right)^{-3},
 \end{split}
 \end{equation}
 where we used that $\omega(y) \leqslant C_\omega e^{C_\omega y}$ (see~\eqref{hyp:rho}) and 
 we assumed that $\Re \sqrt{- \ii k \mu} \geqslant C_\omega$, which is true for 
 $K \geqslant C_\omega^2 / \eta$. 
 
 We now turn towards $\vv_{\inviscid}'''$. When differentiating~\eqref{eq:inviscid-approximation-v} thrice, all terms can be written as $g_j(y) U_s^{(j)}(y)$ where $j=1,2,3$ and $|g_j(y)|\leq C_\eta (1+y)$. Therefore, using assumption \eqref{eq:hyp-uj-2}, we obtain
 \begin{equation} \label{fw.3}
 \| \vv_{\inviscid}''' \|_{L^2(\omega)} \leqslant C_{\inviscid}
 \end{equation}
 for some constant $ C_{\inviscid}$ depending on $\eta$.
 
 Hence, gathering~\eqref{fw.1},~\eqref{fw.2} and~\eqref{fw.3} into~\eqref{def.fw}
 and applying \cref{thm:homo-resolvent-bounds,thm:poincare}, proves that there 
 exists a constant $C_\eta$ such that
 \begin{equation}
 \| \tvv \|_\infty 
 \leqslant
 C_\omega \| \tvv' \|_{L^2(\omega)} 
 \leqslant C_\eta |k|^{-\frac{1}{2}}.
 \end{equation}
 Therefore, we can estimate the difference
 \begin{equation}
 \left| \Phi\left(\mu,k\right)
 - \Phi_{\inviscid}\left(\mu\right) \right|
 = \left| \lim_{y\to \infty} ( \tvv + \vc ) \right|
 \leqslant  C_\eta |k|^{-\frac{1}{2}} + \left|\left(1+\frac{1}{\mu}\right) \frac{1}{\sqrt{-\ii k \mu}}\right|.
 \end{equation}
 which gives the claimed result~\eqref{phi.approx.phi.inv}.
 \end{proof}
 
\subsection{Proof of the persistence of inviscid instabilities}
\label{sec:proofs-123}

We use \cref{thm:convergence-approx-inv} to prove \cref{thm:pdt-sufficient} and \cref{thm:ibl-sufficient}.

Assume that $\Us$ satisfies \cref{crit:pdt-sufficient}. From \cref{thm:winding},
there exists $\eta > 0$ such that the winding number of $\Pinv(\mathcal{C}_\eta)$
round $\ds$ is positive. From \cref{thm:convergence-approx-inv}, the closed curves
$\Phi(\mathcal{C}_\eta,k)$ converge uniformly (say, in the sense of the
Hausdorff distance) to $\Pinv(\mathcal{C}_\eta)$. Hence, for $k$ large enough, 
their winding number around $\ds$ is positive and, from the argument principle
(see e.g. \cite[Chapter~7]{needham1998visual})
there exists $\mu_k$ with $\Im \mu_k > \eta$ such that $\Phi(\mu_k, k) = \ds$.

Assume that $\Us$ satisfies \cref{crit:ibl-global-sufficient}. From \cref{thm:winding},
there exists $\eta > 0$ such that the winding number of $\Pinv(\mathcal{C}_\eta)$
around $0$ is positive. Since it is a smooth closed curve, its winding number around
a small neighborhood of $0$ in the complex plane stays positive. In particular,
it is positive on some interval $[0,\gamma_-]$, where $\gamma_- > 0$. 
By \cref{thm:convergence-approx-inv},
this is still true of $\Phi(\mathcal{C}_\eta,k)$ for large enough $k$. Hence,
there exists $K > 0$ such that, for $k \geqslant K$, and 
$(\sqrt{\nu} k)^{-1} \leqslant \gamma_-$, there exists $\mu_{k,\nu}$ with
$\Im \mu_{k,\nu} > 0$ such that $\Phi(\mu_{k,\nu}, k) = 1/(\sqrt{\nu}k)$.

Assume that $\Us$ satisfies \cref{crit:ibl-local-sufficient}. We proceed as
above remarking that, in this case, the curves wind around a small segment
$[\gamma_-,\gamma_+]$, which leads to instabilities for large enough $k$,
in the range of parameters $\gamma_- \leqslant (\sqrt{\nu} k)^{-1} \leqslant \gamma_+$.

Assume that $\Us''(0) > 0$. From \cref{thm:winding-physical}, there exists
$d, \gamma_- > 0$ such that, for $\gamma \geqslant \gamma_-$, there exists
$\eta(\gamma) > d \gamma^{-2}$ such that $\mathrm{W}_\gamma[\Pinv(\mathcal{C}_{\eta(\gamma)})] = 1$
and $\mathrm{d}_\gamma [\Pinv(\mathcal{C}_{\eta(\gamma)})] \geqslant d \gamma^{-2\kappa}$.
Hence, the difficulty in this case is that the 
security margin between $\Pinv(\mathcal{C}_{\eta(\gamma)})$ and
$\gamma$ depends on $\gamma$.
Assuming that $\nu k^3 \geqslant 1$, one has $k (\sqrt{\nu} k)^{4\kappa} 
\geqslant (k\sqrt{\nu})^{4\kappa-2}$. Since $\kappa < \frac{1}{2}$,
the right-hand side is greater than $(2C/d)^2$ for $k\sqrt{\nu}$
small enough, where $C$ is the constant in \cref{thm:convergence-approx-inv}. Thus $k (\sqrt{\nu} k)^{4\kappa} \geqslant (2C/d)^2$.
Hence $C |k|^{-1/2} \leqslant (d/2) \gamma^{-2\kappa}$ and,
by Rouch\'e's Theorem (see e.g. \cite[Chapter~7]{needham1998visual}), $\mathrm{W}_\gamma[\Phi(\mathcal{C}_{\eta(\gamma)},k)] = 1$.
We can conclude that, for $\sqrt{\nu} k$ small enough and $\nu k^3 \geqslant 1$,
there exists $\mu_{k,\nu}$ with $\Im \mu_{k,\nu} \geqslant d \nu k^2$
such that $\Phi(\mu_{k,\nu},k) = (\sqrt{\nu}k)^{-1}$.
This concludes the proof of \cref{thm:ibl-sufficient} in the physical regime case. 

\subsection{Approximate eigenmodes method}
\label{sec:method2}

As an illustration, we present the approximate eigenmodes method in the case of the prescribed displacement thickness problem. The method consists in two complementary steps. First, one must construct arbitrarily good approximations of the desired eigenmodes. Second, one must control the error between the true solution and the approximate solution. In the context of instability results, this is usually done by a contradiction argument (see e.g. \cite{MR2601044}).

Assume that $\Us$ satisfies \cref{crit:pdt-sufficient}, and let $\mu_0 \in \C$ with $\Im \mu_0 > 0$ such that $\Pinv(\mu_0) = \ds$. We denote by $\vv_0$ the associated inviscid solution defined in~\eqref{eq:inviscid-approximation-v}. Hence $\vv_0(y) \to \ds$ as $y \to +\infty$. Our goal is to build explicitly $\mu$ close to $\mu_0$ and $\vv$ close to $\vv_0$ such that $(\mu,\vv)$ is an arbitrarily precise solution to~\eqref{eq:sol-vv-problem}. We introduce the small parameter $\varepsilon := |k|^{-1/2}$ and consider the following asymptotic expansion:
\[
 \vv(y)  = \sum_{j=0}^{+\infty} \varepsilon^j \left(\vv^\interior_j(y) + \varepsilon \vv^\bl_j \left(\frac{y}{\varepsilon}\right) \right), \quad
 \mu  = \sum_{j=0}^{+\infty} \varepsilon^j \mu_j,
\]
which must solve the equation:
\begin{equation}
 (\mu - \Us) \vv' + \Us' \vv - \ii \varepsilon^2 \vv^{(3)} = F
\end{equation}
together with the boundary conditions:
\[
 \vv(0) = 0,\ 
 \vv'(0) = 1, \ 
 \vv'(y) \to 0, \ 
 \vv(y) \to \ds.
\]
From the impermeability boundary condition, we get that:
\begin{equation}
 \vv^\interior_0(0) = 0 \quad \textrm{and} \quad 
 \vv^\interior_{j+1}(0) = - \vv^\bl_j(0) \quad \text{for } j\geqslant 0.
\end{equation}
From the no-slip boundary condition, we get that, for $j \geqslant 0$:
\begin{equation}
 {\vv^\bl_0}'(0) = 1 - {\vv^\interior_0}'(0) \quad \textrm{and} \quad 
{\vv^\bl_j}'(0) = - {\vv^\interior_j}'(0) \quad \text{for } j\geqslant 0.
\end{equation}
The two boundary conditions at infinity do not concern the boundary
layer terms as these terms are localized near $y = 0$. So they should
only be seen as boundary conditions for the inner inviscid profiles.
We already know that $(\Us,\mu)$ is such that $\vv_0$ matches the leading order
prescribed displacement thickness condition.

\paragraph{Boundary layer profiles}

When deriving the equations satisfied by higher-order profiles, we need to take care to choose the correct expansion of the slowly varying terms. At order $j \geqslant 0$, one has:
\begin{equation}
 \begin{split}
  \mu_0 {\vv^\bl_j}' - \ii {\vv^\bl_j}^{(3)}
  & = \sum_{l=1}^{j} \left(\frac{\Us^{(l)}(0)}{l!} \zeta^l - \mu_l\right) {\vv^\bl_{j-l}}'(\zeta) \\
 & \phantom{=} + \sum_{l=0}^{j-1} \frac{\Us^{(l+1)}(0)}{l!} \zeta^{l} {\vv^\bl_{j-l-1}}(\zeta),
 \end{split}
\end{equation}
with the conventions that, for $j=0$, the sums are empty. These equations are solved iteratively, lifting one boundary condition at each step. The solutions $\vv^\bl_j$ are polynomials times an exponentially decaying profile of the form $z\mapsto \exp\left(-z\sqrt{-\ii \mu_0}\right)$, where the square root is chosen with positive real part.

\paragraph{Interior profiles}

At each step, we must solve:
\begin{equation}
 (\mu_0 - \Us) {\vv^\interior_j}' + \Us' \vv^\interior_j = g_j,
\end{equation}
where $g_0 := F$, $g_1 := - \mu_1 {\vv^\interior_0}'$ and, for $j \geqslant 2$,
\begin{equation}
 g_j := \ii {\vv^\interior_{j-2}}^{(3)} - \sum_{l=1}^{j} \mu_l {\vv^\interior_{j-l}}'.
\end{equation}
Here again, the heart of the structure is unchanged. Only the source term changes. When $g_j$, $g_j'$ and $g_j''$ tend to zero at infinity, then so do ${\vv^\interior_j}'$, ${\vv^\interior_j}''$ and ${\vv^\interior_j}'''$. This ensures the first boundary condition at infinity. To check the prescribed displacement thickness condition, we need to compute $\vv^\interior_j(\infty)$:
\begin{equation}
 \vv^\interior_j(\infty) 
 = 
 \vv^\interior_j(0) \left(1-\frac{1}{\mu_0}\right)
 +
 (\mu_0-1) \int_0^{+\infty} \frac{g_j}{(\mu_0-\Us)^2}.
\end{equation}
Hence, for $j \geqslant 1$, $\vv_j(\infty)$ is the sum of a finite term which we cannot choose plus a term proportional to $\mu_j$:
\begin{equation}
 - \mu_j (\mu_0-1)  \int_0^{+\infty} \frac{{\vv^\interior_0}'}{(\mu_0-\Us)^2}
 = 
 \frac{2 \mu_j}{\mu_0-1} \int_0^{+\infty} \frac{1-\Us}{(\mu_0-\Us)^3}.
\end{equation}
If this coefficient does not vanish, then we can choose $\mu_j$ in order to guarantee that $\vv^\interior_j(\infty) = 0$. We can check that this condition actually correspond to the fact that $\partial \Pinv / \partial \mu (\mu_0) \neq 0$, which is a reasonable condition to be able to apply a local inverse mapping kind of reasoning. Thus, we are able to conduct this construction provided that this derivative does not vanish. 

Let us emphasize that complex analysis methods can allow us to guarantee that we have a $\mu_0$ and $U_s$ such that $\Pinv(\mu_0) = \ds$ and $\partial \Pinv / \partial \mu (\mu_0) \neq 0$. For example, if the winding number of the considered curves around $\ds$ is equal to $1$, the root $\mu_0$ must be simple
(the winding number would be at least equal to $2$ if the derivative vanished
at point $\mu_0$).

\paragraph{Approximate eigenmodes} To construct an approximate eigenmode, we cut the sum at some fixed number $J$. The, we lift the remaining boundary data (which are very small)
with some arbitrary lifting.

\section{Viscosity-induced instabilities for IBL}
\label{sec:strong}

We look for instabilities in the IBL model for which the growth rate $\lambda$ 
would scale like $k^2$. Still denoting by $\mu := \ii \lambda / k$, we look for $\mu$ 
under the form $\mu = \ii \alpha k$, where $\alpha$ depends $\nu$ but is of order $1$ 
with respect to $k$ and $\Re \alpha > 0$ so that there is blow-up.

\subsection{Heuristic approach}

In this regime, we expect that the instabilities will be caused by the viscous term
and that we must thus use another approximation for $\Phi$ than $\Pinv$. As a first
step, we consider the following reduced model of \eqref{eq:sol-vv-problem}
and \eqref{eq:self-consistency-ibl}, in which we have dropped the transport term instead of the viscous one
\begin{equation} \label{eq:sol-vvf-problem}
 \begin{aligned}
  & \mu \vv' - \frac{\ii}{k} \vv''' = F,\\
  & \vv\vert_{y=0} = 0, \quad \vv'\vert_{y=0} = 1,
  \quad \lim_{+\infty} \vv' = 0,
  \quad \lim_{+\infty} \vv = \frac{1}{\sqrt{\nu}k}.
 \end{aligned}
\end{equation}
Up to the first order, recalling that $\mu = \ii \alpha k$, one would drop the viscous
term and introduce the approximation
\begin{equation}
 \vf(y) := - \frac{\ii}{\alpha k} \int_0^y F.
\end{equation}
Since $\int_0^\infty F = 2 \ds$, equating the value of $\vf(\infty)$ with $1/(\sqrt{\nu}k)$
would yield $\alpha = - 2 \ii \sqrt{\nu} \ds$, leading to a purely imaginary eigenvalue
and thus, no blow-up. However, $\vf$ does not match the boundary condition $\vf'(0)=1$
and, as a good approximation to \eqref{eq:sol-vvf-problem}, one must rather use
$\vf + \vc$, where $\vc$ is the boundary corrector 
defined in \eqref{vc.explicit} and solution to
\eqref{system:vc}. We then obtain the following compatibility relation:
\begin{equation}
 \frac{1}{\sqrt{\nu}k} = \vf(\infty) + \vc(\infty) = 
 \frac{-2\ii \Delta_s}{\alpha k}
 + \left(1-\frac{1}{\ii \alpha k}\right) \frac{1}{k \sqrt{\alpha}}.
\end{equation}
Dropping the lower order $1/k^2$ term yields the equation
\begin{equation} \label{eq:f-alpha-nu}
 f(\alpha) = \frac{1}{\sqrt{\nu}},
\end{equation}
where we introduce
\begin{equation}
 f(\alpha) := \frac{-2\ii\ds}{\alpha} + \frac{1}{\sqrt{\alpha}}.
\end{equation}
In the following paragraphs, we prove that \eqref{eq:f-alpha-nu} has
a solution $\alpha_+$ with positive real part and that this construction
is a good approximation for the full system, thus leading to the existence
of an exact eigenmode for \eqref{eq:sol-vv-problem}.

\subsection{Validity of the approximation}

\begin{lemma}
 \label{thm:high-osc-phi-approx}
 Let $\eta > 0$. There exist constants $C$ and $K$ 
 such that for $|k| \geqslant K$ and $\alpha \in \C$ satisfying
 $\Re \alpha \geqslant \eta / k$ and $|\Im \alpha| \geqslant 4 / k$, one has
 \begin{equation} \label{eq:high-osc-phi-approx}
 \left|
 \Phi\left(\ii \alpha k, k\right)
 - \frac{1}{k} f(\alpha)
 \right|
 \leqslant \frac{C}{k^2} \left|\frac{1}{\Im \alpha}\right|
 \left( \left| \frac{1}{\alpha} \right| + \frac{1}{\sqrt{|\alpha|}}
 \right)
 \end{equation}
\end{lemma}

\begin{proof}
 We write $\vv = \vf + \vc + \tvv$, where $\tvv$ is the solution to:
 \begin{equation}
 \begin{aligned}
 & \left(\ii \alpha k - U_s\right) \tvv' + U_s' \tvv
 - \frac{\ii}{k}\tvv'''
 = \tilde F, \\
 & \tvv\vert_{y=0} = 0, \quad \tvv'\vert_{y=0} = 0,
 \quad \lim_{y \rightarrow +\infty} \tvv' = 0,
 \end{aligned}
 \end{equation}
 with
 \begin{equation}
\tilde F:= \Us(\vf+\vc)' - \Us'(\vf+\vc) + \frac{\ii}{k} \vf'''.
 \end{equation}
 Following~\eqref{fw.1} and~\eqref{fw.2}, there exists $C_1$ such that:
 \begin{equation}
 \| \Us' \vc \|_{L^2(\omega)} + 
 \| \Us \vc' \|_{L^2(\omega)}
 \leqslant C_1 |\alpha k^2|^{-\frac{1}{2}}.
 \end{equation}
 From the assumptions on $\Us$ and $F$, there exists $C_2$ such that:
 \begin{equation}
 \| \Us' \vf \|_{L^2(\omega)} + 
 \| \Us \vf' \|_{L^2(\omega)} + 
 \frac{1}{k} \| \vf''' \|_{L^2(\omega)}
 \leqslant
 C_2 \left| \alpha k \right|^{-1}
 \end{equation}
 Choosing $K$ such that \cref{thm:homo-resolvent-bounds} applies
 and combining estimate~\eqref{energy:v'.2} with \eqref{estimate:l2rho_linfty} from \cref{thm:poincare} 
 concludes the proof of~\eqref{eq:high-osc-phi-approx}, where we have enlarged $C$
 in order to incorporate the $k^{-2}$ term from the approximation
 in the right-hand side.
\end{proof}

\subsection{Existence of roots with positive real part}

Since $\Us$ is monotone, $\ds \neq 0$. If $\nu > 0$ is small enough, 
we can solve \eqref{eq:f-alpha-nu} and find two roots
\begin{equation}
 \alpha_{\pm} = \nu
 \left(1 \pm \sqrt{1 - \frac{8 \ii \ds}{\sqrt{\nu}}} -
 \frac{4 \ii \ds}{\sqrt{\nu}}
 \right).
\end{equation}
Hence, there exists a root $\alpha_+$ with positive real part, 
depending on $\nu$, whose asymptotic behavior as $\nu$ tends
to zero is
\begin{equation}
\Re \alpha_{+} \sim 2 \nu^{3/4} \sqrt{|\ds|}
\quad \text{and} \quad
\Im \alpha_{+} \sim - 4 \nu^{1/2} \ds.
\end{equation}
We prove that, under the assumptions of \cref{thm:ibl-strong},
this unstable root also exists for the full initial system.
For a small enough $\epsilon$ there exist positive constants $c_i$
and $c_r$ such that, for small enough $\nu$ it holds for all $\alpha$
with $|\alpha_+ - \alpha| \leqslant \nu^{3/4} \epsilon$ that
\begin{equation}
\Re\, \alpha \geqslant c_r \nu^{3/4} 
 \quad \text{and} \quad|\Im\, \alpha| \geqslant c_i \nu^{1/2}
 .
\end{equation}
Applying \cref{thm:high-osc-phi-approx} with $\eta = c_r$, the estimates
apply for $k \nu^{3/4}$ large enough and we obtain the bound
\begin{equation}
 k\left|\Phi(\ii\alpha k,k) - \frac{1}{k}f(\alpha)\right| \leqslant \frac{C}{k \nu} .
\end{equation}
By construction $f(\alpha_{+}) = 1/\sqrt{\nu}$. Performing an asymptotic expansion of $f$ around $\alpha_+$,  for small enough $\nu$ we
can find a constant $c$ depending on $\eps$ such that, for $\theta \in [0,2\pi]$,
\begin{equation}
 \left|f(\alpha_+ + \epsilon \nu^{3/4} \ee^{\ii \theta}) - \frac{1}{\sqrt{\nu}}\right| \geqslant  c\, \nu^{-1/4}.
\end{equation}
Hence, if $k \nu^{3/4} > C/c$, the two closed curves
\begin{align}
\theta & \mapsto f(\alpha_+ + \epsilon \nu^{3/4} \ee^{\ii \theta}), \\
\theta & \mapsto k \Phi(\ii \alpha_+ k + \epsilon k\nu^{3/4} \ee^{\ii \theta}, k)
\end{align}
have the same winding number around $1/\sqrt{\nu}$, thanks to
Rouch\'e's Theorem (see e.g. \cite[Chapter~7]{needham1998visual}).
Thus, we can conclude 
the existence of a root $\alpha_{k,\nu}$, with
$\Re \alpha_{k,\nu} \geqslant c_r \nu^{3/4}$.
This concludes the proof of \cref{thm:ibl-strong}
under the assumption $\ds \neq 0$.

\begin{remark}
 In the context of this paper, we only consider monotone shear flows
 for simplicity. However, the construction presented in this section
 makes no use of the monotonicity assumption, except for the fact
 that $\ds \neq 0$. Hence, \cref{thm:ibl-strong} in fact holds for any
 smooth enough shear flow satisfying $\ds \neq 0$. Moreover,
 only very particular profiles satisfy $\ds = 0$ and this possibility 
 is excluded for a large realistic class of profiles.
 Indeed, let $\Us \in C^0(\R_+,\R)$ with $\Us(0) = 0$ and $\Us(y) \to 1$ as 
 $y \to +\infty$ be such that $\Us(y) \leqslant 1$ on $\R_+$. Then $\ds > 0$, 
 as the integral of a non-negative non identically vanishing function.
\end{remark}

\subsection{About the Tollmien-Schlichting instability}
\label{sec:TS}

The  Tollmien-Schlichting instability (see \cite{DrazinReid} for a thorough physical description and \cite{MR3566199,MR3464020} for a mathematical analysis) is also a viscosity induced instability at the level of the Navier-Stokes equations. It takes place in the regime 
\begin{equation} \label{eq:scalingTS}
k \propto \nu^{-3/8}, \quad \Im \mu \propto \nu^{1/8}, \quad \Re \mu \propto \nu^{1/8},
\end{equation}
and is therefore not described by the results of Theorems \ref{thm:ibl-sufficient} and \ref{thm:ibl-strong}. However, we expect it to be present within the IBL model, and we give here a short formal derivation for the sake of completeness.

As above, we look for an expansion of the form $\vv= \vv_\inviscid + \vv_c + \tvv$, where
\begin{itemize}
\item $\vv_\inviscid$ satisfies the inviscid equation 
$$ (\mu - U_s) \vv_\inviscid' + U'_s \vv_\inviscid = F $$
together with the conditions at infinity: 
$ \vv_\inviscid(\infty) = \frac{1}{\sqrt{\nu} k}, \quad \vv_\inviscid'(\infty) = 0.$
\item $\vv_c$ is a boundary layer corrector, localized near the boundary ($\vv_c$ and its derivatives should decay fast at infinity), built so that 
\begin{equation}  \label{eq:bcTS}
\vv_\inviscid(0) + \vv_c(0) \approx 0, \quad \vv_\inviscid'(0) + \vv_c'(0) \approx 1. 
\end{equation}
\item $\tvv$ is a remainder term.
\end{itemize} 
Note that we have slightly modified the definition of  the inviscid term $\vv_\inviscid$. This is due to the fact that we want our boundary layer corrector $\phi_c$ to be negligible outside a  neighbourhood of zero, whereas in our previous derivation $\phi_c$ was constant outside the boundary layer. Hence we need the inviscid part of the solution to capture the whole boundary condition at infinity. As a consequence, the inviscid term $\vv_\inviscid$ constructed in this paragraph differs from our previous definition by the addition of a term $C(\mu-\Us)$ for a suitable constant $C$. 

The inviscid part of the approximation is computed explicitly: we find 
$$ \vv_\inviscid(y) = \frac{1}{\sqrt{\nu}k} \frac{\mu-U_s}{\mu-1} - (\mu-U_s) \int^{+\infty}_y \frac{F(y')}{(\mu - U_s(y'))^2} \dd y'. $$ 
Note that as $\mu \rightarrow 0$, 
$$ \int_0^{+\infty}  \frac{F(y')}{(\mu - U_s(y'))^2} \dd y' \sim -\frac{1}{\mu \beta}, $$
where $\beta := U'_s(0) > 0$. Assuming that $|\mu|\ll 1$, it follows that 
\begin{equation} \label{eq:vvinv_0}
\vv_\inviscid(0) \approx -\frac{\mu}{\sqrt{\nu}k} + \frac{1}{\beta}, \quad   \vv_\inviscid'(0) \approx \frac{\beta}{\sqrt{\nu}k}.
\end{equation}
Note that in the Tollmien-Schlichting regime \eqref{eq:scalingTS}, we expect $\vv_\inviscid(0)$ to be of order~$1$, while $\vv_\inviscid'(0)$ is of order $\nu^{-1/8}$. 

We then look for a boundary layer corrector of the form 
$\vv_c(y)= W(\delta^{-1}y)$, where  $0<\delta\ll 1$ is the size of the boundary layer. As we want $\vv_\inviscid + \vv_c + \tvv$ to satisfy \eqref{eq:sol-vv-problem}, we expect that 
$$ (\mu - U_s) \vv_c' + U'_s \vv_c  - \frac{i}{k} \vv_c^{(3)} \approx 0. $$ 
After differentiation with respect to $y$, we get: 
$$ (\mu - U_s) \vv_c'' + U''_s \vv_c  - \frac{i}{k} \vv_c^{(4)} \approx 0. $$ 
We plug the expression for $\vv_c$ in this equation, to find
\begin{equation}
\label{eq:W-noapprox}
\left(\frac{\mu}{\delta^2}-\beta \frac{\xi}{\delta}\right) W''(\xi) + U_s''(0) W (\xi) - \frac{i}{k\delta^4} W^{(4)}(\xi) \approx 0.
\end{equation}
Note that we used the approximations $U_s(y) \approx \beta y = \beta \delta \xi$ and $U''_s(y) \approx U''_s(0)$, valid in the boundary layer.  Setting $\delta:=(k\beta)^{-1/3}$, we observe that we can further neglect the term $U_s''(0)  W(\xi)$ in \eqref{eq:W-noapprox} in the asymptotics $|k|\to \infty$. Hence, $W''$ is the solution of the Airy-type equation
\begin{equation*}
(-i\mu k^{1/3} \beta^{-2/3} + i \xi) W''(\xi) - W^{(4)}(\xi)=0.
\end{equation*}
We obtain
\begin{equation*}
W''(\xi)= C \Ai (\eta + e^{i\pi/6} \xi),
\end{equation*}
where $\eta:= -k^{1/3}\beta^{-2/3} \mu e^{i\pi/6}$ and $\Ai$ is the Airy function of the first kind. From the decay condition at infinity, we deduce that 
$$ W(\xi) = C e^{-i\pi/3} \Ai(\eta + e^{i \pi/6} \xi, 2) $$ 
where $\Ai(\cdot,k)$ is the $k$-th antiderivative of $\Ai$ that decays along the ray $e^{i \pi/6} \mathbb{R}_+$. 

Eventually, we use this expression in \eqref{eq:bcTS}, which yields: 
$$ \vv_\inviscid(0)  + C e^{-i\pi/3} \Ai(\eta, 2) \approx 0, \quad \vv_\inviscid'(0) + C \frac{e^{-i\pi/6}}{\delta} \Ai(\eta, 1) \approx 1. $$
In the regime \eqref{eq:scalingTS}, the $1$ at the right-hand side of the second equation can be neglected, and we find 
$$ \frac{\vv_\inviscid(0)}{\vv_\inviscid'(0)} \approx e^{-i\pi/6} \delta \frac{\Ai(\eta,2)}{\Ai(\eta,1)} $$
Substituting the expressions in \eqref{eq:vvinv_0} leads to
$$ -\frac{\mu}{\beta} \left(1 - \frac{\sqrt{\nu}k}{\mu \beta}\right) \approx e^{-i\pi/6} \delta \frac{\Ai(\eta,2)}{\Ai(\eta,1)} $$
or using the definition of $\eta$:
$$ \left(1 - \frac{\sqrt{\nu}k}{\mu \beta}\right) \approx \frac{\Ai(\eta,2)}{\eta \Ai(\eta,1)} $$
We recognize here the leading order of the Tollmien-Schlichting dispersion relation: see \cite[chapter 28]{DrazinReid}, equations (28.6)-(28.7)-(28.15)-(28.30).


\section{Construction of unstable shear flows}
\label{sec:construction}

The previous sections of this paper concerned a ``forward'' problem:
given a shear flow profile, does it exhibit instabilities? Here, we adopt
a kind of ``backwards'' point of view, described by Theorem \ref{thm:construction}: given an asymptotic behavior of
the eigenvalue (for large $k$) 
of the form $- \ii \mu k$, with $\mu$ in some spectral
domain, we build shear flows displaying such instabilities.

In order to prove \cref{thm:construction}, we proceed in two steps. First,
we construct shear flows satisfying the inviscid conditions of the form
$\Pinv(\mu) = \ds$ (see \cref{sec:construction1}), 
$\Pinv(\mu) = 0$ (see \cref{sec:construction2}) and
$\Pinv(\mu) = \gamma > 0$ (see \cref{sec:construction3}).
We prove that these constructions are possible for $\mu$ lying 
within some spectral domains. 
By analogy with the three criteria, we name our spectral domains 
$\Gamma_1, \Gamma_2$ and $\Gamma_3$ for each of the three situations
(see \cref{fig:domains}). Then, we prove that these shear flows
lead to instabilities at the viscous level 
with the claimed asymptotic behavior
(see \cref{sec:construction-viscous}).

Eventually, we prove that the explicit example shear flow 
given in \cref{rk:examples} indeed satisfies \cref{crit:pdt-sufficient} and \cref{crit:ibl-global-sufficient} (see \cref{ssec:crit2-spec}).

\begin{figure}[!ht]
 \centering
 \includegraphics[width=10cm]{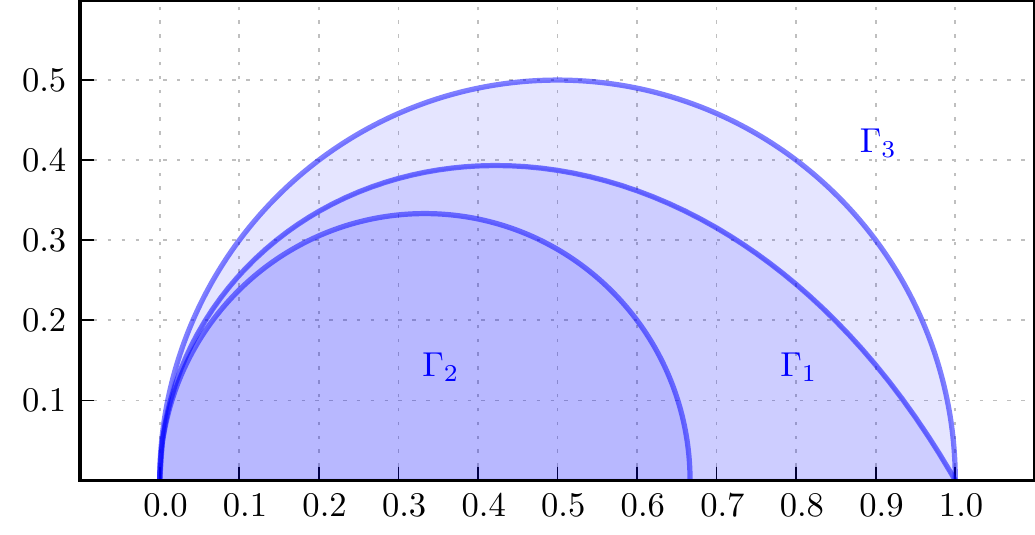}
 \caption{Illustration of the spectral domains $\Gamma_1, \Gamma_2$ and $\Gamma_3$.}
 \label{fig:domains}
\end{figure}

\subsection{First spectral domain}
\label{sec:construction1}

We introduce the following spectral domain
\begin{equation}
 \label{eq:def-gamma1}
 \Gamma_1 := \left\{
 a + \ii b \in \C, \enskip 0 < a < 1, \enskip
 0 < b < \sqrt{\frac{3a(1-a)^2}{(4-3a)}}
 \right\}.
\end{equation}
The domain $\Gamma_1$ is only slightly smaller with respect to $b$ 
than the necessary condition from \cref{thm:possible-mu} and we suspect 
that the given $\Gamma_1$ is optimal.
We intend to prove that, for any $\mu \in \Gamma_1$, we can build
a monotone shear flow $\Us$ with nice decay properties such that 
$\Pinv(\mu) = \ds$. 

Recalling the definition of $\Pinv(\mu)$ and $F$ and integrating by parts, we have
\begin{eqnarray*}
\Pinv(\mu)&=&(\mu-1)\int_0^\infty\frac{1-\Us(y) + y \Us'(y)}{(\mu-\Us(y))^2}\:dy\\
&=&\int_0^\infty\frac{1-\Us(y)}{(\mu-\Us(y))^2}\left( 2\mu-1-\Us(y)\right)\:dy.
\end{eqnarray*}
Letting
$\mu = a + \ii b$, we rewrite the integral above as
\begin{equation} \label{eq:phi-inv-p-mu}
\Phi_{\inviscid}(\mu) = \int_0^\infty
\frac{1-U_s(y)}{|\mu - U_s(y)|^4} p_\mu(U_s(y))\,\dd y
\end{equation}
with
\begin{equation} \label{eq:p-mu}
p_{\mu}(x): = (2 \mu - 1 - x) (\conj{\mu} -x)^2
\end{equation}
so that
\begin{align} 
\label{eq:p-mu-re}
\Re p_{\mu}(x) &= (2a - 1 - x) (a-x)^2 + b^2 (2a + 1 -3x) \\
\label{eq:p-mu-im}
\Im p_{\mu}(x) &= 2b \left[ (a-x)(1-a) - b^2 \right].
\end{align}
In the  expression \eqref{eq:phi-inv-p-mu} for $\Phi_{\inviscid}$ the integrand depends
only on $U_s(y)$ and not on $y$ directly. Hence by studying the image
of $p_{\mu}([0,1])$ we can determine the possible range of $\mu$.
In this system the condition for an eigenmode is
$\Phi_{\inviscid}(\mu) = \ds$ and recalling
$\ds = \int_0^\infty [1-U_s(y)] \dd y$ the condition can be
rewritten as
\begin{equation*}
0 = \int_0^\infty
\frac{1-U_s(y)}{|\mu - U_s(y)|^4} q_\mu(U_s(y))\,\dd y
\end{equation*}
with
\begin{equation*}
q_{\mu}(x) := |\mu - x|^4 - p_{\mu}(x).
\end{equation*}
Adapting the proof of \cref{thm:possible-mu}, it is clear that an
eigenmode can only exist if $q_{\mu}([0,1])$ is not restricted
to a half-plane. For the spectral domain $\Gamma_1$ from
\eqref{eq:def-gamma1}, we can show that $q_{\mu}([0,1])$ contains all
directions.

\begin{lemma}
 \label{thm:qmu123}
 For any $\mu \in \Gamma_1$, there exists $0 \leqslant x_1 < x_2 < x_3 \leqslant 1$
 such that:
 \begin{equation} \label{hyp.a123.qmu}
 \left\{ \alpha_1 q_\mu(x_1) + \alpha_2 q_\mu(x_2) + \alpha_3 q_\mu(x_3),
 \enskip (\alpha_1,\alpha_2,\alpha_3) \in \R_+^3 \right\}
 = \C.
 \end{equation}
\end{lemma}

\begin{proof}
 The proof relies on a careful study of the
 position of the curve $q_\mu([0,1])$ within the complex plane. In
 particular, we identify three directions spanning the whole complex
 plane. To this aim, we write the curve as
 \begin{equation}
 q_\mu(x) = A_\mu(x) + \ii B_\mu(x),
 \end{equation}
 where we define
 \begin{align}
 \label{def.Amu}
 A_\mu(x) & := (x-a)^2 P_\mu(x) + b^2 \left[2x^2 + (3-4a)x+(2a^2-2a-1) \right] + b^4, \\
 \label{def.Bmu}
 B_\mu(x) & := 2b \left[ (1-a)(x-a) + b^2 \right], \\
 \label{def.Pmu}
 P_\mu(x) & := x^2 + (1-2a) x + (1-a)^2.
 \end{align}
 
 \bigskip \noindent \textbf{Curves cross the negative real axis.}
 First, notice that there exists a unique $x_0\in \R$ such that $B_\mu(x_0)= \Im (q_\mu(x_0))=0$, given by
  \begin{equation}
 x_0 = \frac{a(1-a)-b^2}{1-a}.
 \end{equation}
 Since $\mu \in \Gamma$, from~\eqref{eq:def-gamma1}, $b^2 < a(1-a)$. Hence $0<x_0<a<1$. 
  We are now interested in the sign of $A_\mu(x_0)$. After simplification, one obtains
 \begin{equation}
 \begin{split}
 A_\mu(x_0)
 & = \frac{b^2}{(1-a)^4} \bigg[ b^6
 + b^4 (1-a)(1-2a)
 - b^2 (1-a)^3 (1+a)
 - (1-a)^5 \bigg] \\
 & = \frac{b^2}{1-a} T_a\left(\frac{b^2}{a(1-a)}\right),
 \end{split}
 \end{equation}
 where we introduce
 \begin{equation}
 T_a(\tau) :=
 \tau^3 a^3  + \tau^2 a^2 (1-2a)
 - \tau a (1-a^2) - (1-a)^2.
 \end{equation}
 We must study the sign of $T_a$ on $[0,1]$.
 For any $a \in (0,1)$, $T_a$ is a polynomial of degree three which behaves like $T_a(\tau) \sim \tau^3 a^3$ when $\tau \to \pm \infty$. Moreover, one checks that
 \begin{align}
 T_a(0) & = -(1-a)^2 < 0, \\
 T_a(1) & = -(1-a) < 0, \\
 T_a'(0) & = - a (1-a^2) < 0.
 \end{align}
 Since $T_a'(0) < 0$ and $T_a$ tends to $-\infty$ at $-\infty$, $T_a$ has a local maximum in $(-\infty,0)$. If there existed $\tau_+ \in (0,1)$ such that $T_a(\tau_+) \geq 0$, then $T_a$ would have a local maximum within $(0,1)$ because both $T_a(0)$ and $T_a(1)$ are negative. Since $T_a'(0) < 0$, $T_a$ would also have a local minimum within $(0,\tau_+)$. But $T_a'$ is a polynomial of degree two and cannot have three distinct roots. So $T_a(\tau) < 0$ on $(0,1)$ for any $a \in (0,1)$. This yields the conclusion: $\Re q_\mu(x_0) < 0$.
 
 \bigskip \noindent
 \textbf{Curves cover more than a half-plane.}
 We now prove the following inequality, which we will use in the next paragraph:
 \begin{equation} \label{PDT.ineq.bqq}
 \Im\left( q_\mu(0) \overline{q_\mu(1)}\right) < 0.
 \end{equation}
 Plugging in the definitions~\eqref{def.Amu},~\eqref{def.Bmu} and~\eqref{def.Pmu}, factoring and simplifying by positive terms, we obtain that inequality~\eqref{PDT.ineq.bqq} is equivalent to:
 \begin{equation} \label{PDT.ineq.bqq2}
 (4-3a) b^4 + 2 (1-a)^2 (2-3a) b^2 - 3 a(1-a)^4 < 0.
 \end{equation}
 Seeing~\eqref{PDT.ineq.bqq2} as a second order polynomial in $b^2$, this inequality is satisfied if and only if $b^2$ is smaller than its positive root. Hence, the condition amounts to:
 \begin{equation}
 b^2 < (1-a)^2 \frac{3a}{4-3a},
 \end{equation}
 which is precisely the definition of the set $\Gamma_1$.
 
 \bigskip \noindent
 \textbf{Existence of generating directions.}
 Let us assume that $b > 0$. First, from~\eqref{def.Amu} and~\eqref{def.Bmu},
 $q_\mu(1)$ lies within the upper-right quarter plane. Second, thanks
 to~\eqref{PDT.ineq.bqq}, $q_\mu(0)$ lies under the diagonal line
 passing through $q_\mu(1)$. Third, $\Im q_\mu(0) < 0$. 
 
 Since $q_\mu(x_0) \in (-\infty, 0)$, it follows immediately that 0 belongs to the interior of the convex envelope of the points $q_\mu(0)$, $q_\mu(x_0)$ and $q_\mu(1)$.
 This concludes the proof of \cref{thm:qmu123}.
\end{proof}

When $q_{\mu}$ spans the whole complex plane, we can construct a
shear flow profile $U_s$ by concentrating $U_s$ on appropriate values. 
This follows immediately from the following lemma by choosing $\varphi=q_\mu$.

\begin{lemma}
 \label{thm:phi123}
 Let $\mu \in \C$ with $\Im \mu \not = 0$ and
 $\varphi \in C^0([0,1], \C)$. Assume that there exists
 $0 \leq x_1 < x_2 < x_3 \leq 1$ such that:
 \begin{equation} \label{hyp.a123}
 \left\{ \alpha_1 \varphi(x_1) + \alpha_2 \varphi(x_2) + \alpha_3 \varphi(x_3),
 \enskip (\alpha_1,\alpha_2,\alpha_3) \in \R_+^3 \right\}
 = \C.
 \end{equation}
 Then there exists a smooth increasing shear flow profile $U_s$ with
 $U_s(0) = 0$ converging exponentially to $1$ at infinity such that:
 \begin{equation} \label{ccl.a123}
 \int_0^{\infty} \frac{(1-U_s(y))}{|\mu - x|^4}\, \varphi(U_s(y))\, \dd y = 0.
 \end{equation}
\end{lemma}

\begin{proof}
 Let $\theta \in \mathcal{C}^\infty(\R,\R_+)$ be a smooth function,
 compactly supported in $[-1,1]$ with $\int \theta = 1$.  By
 continuity of $\varphi$, we can assume that $0 < x_1$ and $x_3 < 1$,
 while preserving assumption~\eqref{hyp.a123}.  For $1 \leq i \leq 3$
 and $0 < \eta < \min(x_1,1-x_3)$, we define
 \begin{equation} \label{def.hi}
 h_{i, \eta}(x) := \frac{1}{\eta} \theta\left(\frac{x-x_i}{\eta}\right).
 \end{equation}
 We also introduce the following complex numbers:
 \begin{equation}
 \psi_{i,\eta} := \int_0^1 \frac{(1-x)}{|\mu-x|^4} \varphi(x) h_{i,\eta}(x) \dd x.
 \end{equation}
 Since $\varphi$ is continuous and the $h_{i,\eta}$ are approximations of
 Dirac masses located at the $x_i$, we deduce from the
 assumption~\eqref{hyp.a123} that there exists $\eta > 0$ small enough such
 that
 \begin{equation} \label{hyp.apsi123}
 \left\{ \alpha_1 \psi_{1,\eta} + \alpha_2 \psi_{2,\eta} + \alpha_3 \psi_{3,\eta},
 \enskip (\alpha_1,\alpha_2,\alpha_3) \in \R_+^3 \right\}
 = \C.
 \end{equation}
 With $\eta$ fixed, we now define
 \begin{equation}
 h(x) := \alpha_1 h_{1,\eta}(x) + \alpha_2 h_{2,\eta}(x) + \alpha_3 h_{3,\eta}(x)
 + \frac{1}{1-x},
 \end{equation}
 where $\alpha_1, \alpha_2$ and $\alpha_3$ are positive real numbers
 chosen such that:
 \begin{equation} \label{eq.hsumnull}
 \int_0^1 \frac{(1-x)}{|\mu-x|^4} \varphi(x) h(x) \dd x = 0.
 \end{equation}
 Indeed, from~\eqref{hyp.apsi123},
 $\int_0^1 \varphi(x) |\mu -x|^{-4} \dd x$ can be expressed as a
 positive linear combination of the $\psi_{i,\eta}$.
 From~\eqref{def.hi}, we have $h(x) > 0$ for any $x \in [0,1)$.  We
 now define
 \begin{equation} \label{def.H}
 H(x) := \int_0^x h(x') \dd x'.
 \end{equation}
 Since $h > 0$ and $h(x) \sim_{x\to 1^-} (1-x)^{-1}$, $H$ defines a continuous
 increasing bijection from $[0,1)$ to $[0,+\infty)$. Moreover, for
 $x > x_3 + \eta$ (we assumed that $x_3 + \eta < 1$), we have
 \begin{equation} \label{eq.H.asympt}
 H(x)
 = \int_0^{x_3+\eta} h(x') \dd x' + \int_{x_3+\eta}^{x} \frac{\dd x'}{1-x'}
 = H_0 - \ln(1-x),
 \end{equation}
 where $H_0 \in \R_+$ is a fixed constant. We now define the shear-flow profile
 as
 \begin{equation} \label{def.UH}
 U_s(y) := H^{-1}(y).
 \end{equation}
 One checks that $U_s$ is an increasing function, with $U_s(0) = 0$ and
 $U_s(y) \to 1$ as $y \to +\infty$. From~\eqref{eq.H.asympt},
 for $y$ large enough:
 \begin{equation}
 U_s(y) = 1 - \ee^{H_0} \ee^{-y}.
 \end{equation}
 Moreover, the key spectral condition \eqref{ccl.a123} is satisfied. Indeed:
 \begin{equation}
 \int_0^{+\infty} \frac{(1-U_s(y))}{|\mu - \Us(y)|^4} \varphi(U_s(y)) \dd y
 = \int_0^1  \frac{(1-x)}{|\mu-x|^4} \frac{\varphi(x)}{U_s'(U_s^{-1}(x))} \dd x.
 \end{equation}
 From~\eqref{def.H} and~\eqref{def.UH}, $U_s'(U_s^{-1}(x))=1/h(x)$.
 Thus,~\eqref{ccl.a123} follows from~\eqref{eq.hsumnull}.
\end{proof}

\subsection{Second spectral domain}
\label{sec:construction2}

We introduce the following spectral domain
\begin{equation} \label{eq:def-gamma2}
 \Gamma_2 :=
 \left\{ a + \ii b \in \C, 0<a<\frac{2}{3}, 0<b<\sqrt{\frac{a(2-3a)}{3}} \right\}.
\end{equation}
For any $\mu \in \Gamma_2$, we prove that 
there exists a shear flow such that $\Pinv(\mu) = 0$. 
Let $x_0$ be such that $\Im p_\mu(x_0) = 0$. Using \eqref{eq:p-mu-im},
$x_0 - a = - b^2 / (1-a)$ and $x_0 \in (0,1)$. At this $x_0$, one has:
\begin{equation}
 \Re p_\mu(x_0) = \frac{b^2}{(1-a)^3} \left(b^2 + (1-a)^2\right)^2.
\end{equation}
Hence, there exists $x_0 \in (0,1)$ with $p_\mu(x_0) \in \R_+$. 
Then, we consider the condition
\begin{equation}
 \Im \left( p_\mu(1) \overline{p_\mu(0)} \right) > 0.
\end{equation}
We obtain that it is satisfied if $\mu \in \Gamma_2$. In such a case, $\Re p_\mu(1) < 0$, $\Im  p_\mu(1) < 0$, $\Im  p_\mu(0) > 0$, and
$0$ is contained within the interior of the triangle $p_\mu(0),p_\mu(x_0)$
and $p_\mu(1)$. Hence, \cref{thm:phi123} applied with $\varphi = p_\mu$ 
yields the existence of the claimed profile.

\subsection{Third spectral domain}
\label{sec:construction3}

We introduce the following spectral domain
\begin{equation} \label{eq:def-gamma3}
 \Gamma_3 := \left\{ a + \ii b \in \C, 0 < a < 1, 0 < b < \sqrt{a(1-a)} \right\}.
\end{equation}
Let $\gamma > 0$ and $\mu \in \Gamma_3$. 
We prove that there exists a shear flow such 
that $\Pinv(\mu) = \gamma$.
We know that there exists $x_0 \in (0,1)$ such that $p_\mu(x_0) \in \R^*_+$.
Moreover:
\begin{equation}
 \Im(p_\mu'(x_0)) = - \frac{2 b (1-a)}{((x_0-a)^2+b^2)^2} \neq 0.
\end{equation}
Hence, there exist $x_\pm \in (0,1)$ with $\Re (p_\mu(x_\pm)) > 0$ 
and $\pm \Im(p_\mu(x_\pm)) > 0$.

\bigskip 

We turn to the construction of $\Us$. Let $\theta \in C^\infty(\R,\R_+)$,
compactly supported in $[-1,1]$ with $\int \theta = 1$. 
For $\eps>0$ small, define
\begin{align}
 h_\pm(x) & := \frac{1}{\epsilon} \theta \left(\frac{x-x_\pm}{\epsilon}\right), \\
 \psi_\pm & := \int_0^1 (1-x) |\mu-x|^{-4} p_\mu(x) h_\pm(x) \dd x.
\end{align}
Hence $h_\pm \geqslant 0$ and $\psi_\pm \in \C$. If $\epsilon > 0$ is small enough,
one has $\Re(\psi_\pm) > 0$ and $\pm \Im(\psi_\pm) > 0$ because we have
chosen approximations of unity. We also choose $\epsilon$ such that
$x_\pm -\epsilon >0$ and $x_\pm + \epsilon < 1$.
We choose
\begin{equation}
 h(x) := \alpha_- h_-(x) + \alpha_+ h_+(x) + \frac{\beta}{1-x},
\end{equation}
where $\alpha_\pm ,\beta > 0$ need to be chosen. Hence
\begin{equation}
 \int_0^1 (1-x) |\mu-x|^{-4} p_\mu(x) h(x) \dd x = \alpha_- \psi_- + \alpha_+ \psi_+ + \beta \psi,
\end{equation}
where $\psi \in \C$ is defined as $\psi := \int_0^1 |\mu-x|^{-4} p_\mu(x)\dd x$.
We try to find positive coefficients such that
\begin{equation}
 \alpha_- \psi_- + \alpha_+ \psi_+ + \beta \psi = \gamma.
\end{equation}
By symmetry, we assume that $\Im(\psi) \geqslant 0$. Let $\beta, \sigma > 0$. We define
\begin{align}
 \alpha_+ & := - \sigma \beta \Im(\psi_-) > 0, \\
 \alpha_- & := \sigma \beta \Im(\psi_+) + \beta \frac{\Im(\psi)}{-\Im(\psi_-)} > 0.
\end{align}
This ensures that $\Im(\alpha_- \psi_- + \alpha_+ \psi_+ + \beta \psi ) = 0$.
Moreover
\begin{equation}
 \begin{split}
 \Re \left(\alpha_- \psi_- + \alpha_+ \psi_+ + \beta \psi \right) 
  & = \beta  \Big[ \Re(\psi) 
 +  \frac{\Im(\psi)\Re(\psi_-)}{-\Im(\psi_-)} \\
 & + \sigma \left( \Im(\psi_+)\Re(\psi_-) - \Im(\psi_-) \Re(\psi_+)
 \right) \Big]
 \end{split}
\end{equation}
For $\sigma$ large enough, the term between the brackets is positive because 
by construction $\Im(\psi_+)\Re(\psi_-) - \Im(\psi_-) \Re(\psi_+) > 0$.
Once such a $\sigma$ is fixed, one can choose $\beta > 0$ such that the 
product is equal to $\gamma$. Therefore, we have built an $h$ such that:
\begin{equation}
\int_0^1 (1-x) |\mu-x|^{-4} p_\mu(x) h(x) \dd x = \gamma.
\end{equation}
From $h > 0$, one defines $\hs(x) := \int_0^x h(x')\dd x'$
and then $\Us(z) := \hs^{-1}(z)$. The shear flow then satisfies
the integral spectral condition. Moreover, for
$z$ large enough, one has:
\begin{equation}
 \Us(z) = 1 - \exp\left(\frac{H_0}{\beta}\right) \exp\left(-\frac{z}{\beta}\right),
\end{equation}
where $H_0$ is a fixed constant, so all decay properties are satisfied.
This concludes the proof of \cref{thm:construction} for the IBL model.

\subsection{Construction of viscous eigenmodes}
\label{sec:construction-viscous}

At this stage, for each spectral domain $\Gamma_i$ for $i=1,2,3$ and for any $\mu \in \Gamma_i$, we are able to construct a shear flow $\Us$ such that the associated $\Pinv$ satisfies $\Pinv(\mu)=\gamma_i$, with $\gamma_1=\Delta_s$, $\gamma_2=0$ and $\gamma_3$  an arbitrary positive number.
At a formal level, these shear flows are expected to satisfy Criterion $i$. However, this might be lengthy to
prove due to possible corner cases. Moreover, even if we succeeded, 
applying \cref{thm:pdt-sufficient} or \cref{thm:ibl-sufficient} 
would not be sufficient to obtain the precise asymptotic behavior 
of the eigenvalues claimed in \cref{thm:construction}. We explain here
how to circumvent both problems simultaneously using once again complex analysis
arguments.

\paragraph{The PDT problem}

We start with the easier case of the prescribed displacement thickness
problem. Let $\mu \in \Gamma_1$. Thanks to \cref{sec:construction1}, there
exists a smooth monotone $\Us$ such that $\Pinv(\mu) = \ds$.
Since $\Pinv$ is holomorphic near $\mu$ and non constant,
there is a non vanishing derivative of $\Pinv$ at $\mu$. Thus, there 
exists $c\in \C$ and $n \in \N^*$ such that, 
one has $\Pinv(\mu+z) -\ds \sim c z^n$ as $z \to 0$.
In particular, $\Pinv$ maps small enough circles 
$\mu+r \ee^{\ii\theta}$ to closed curves close to small circles $\ds + c r^n \ee^{n \ii \theta}$
that have a positive winding number around $\ds$.
Once the circle radius is fixed (small enough such that
the local Taylor approximation holds), we
can consider a sequence of radiuses $r_k := r |k|^{-\alpha}$.
From \cref{thm:convergence-approx-inv}, the distance between
the viscous curve $\Phi(\mu+r_k \ee^{\ii \theta},k)$
and the inviscid curve $\Pinv(\mu+r_k\ee^{\ii\theta})$ is
bounded by $Ck^{-1/2}$. Hence, by Rouch\'e's Theorem (see e.g. \cite[Chapter~7]{needham1998visual}), it has a positive winding number around $\ds$ and we can conclude to the existence of eigenvalues $\lambda_k = -\ii k \mu_k$ with $\mu_{k} \to \mu$
provided that $k^{1/2-n\alpha} \geqslant 2C/(|c| r^n)$.
For example, we can choose $\alpha = 1/(4n)$.

\paragraph{The IBL problem}

We use the same idea. The small difference is that we want to allow
a range of parameters $\gamma = (\sqrt{\nu}k)^{-1}$ which is not reduced
to a single point $\ds$ but is a small interval. This prevents us from
having the convergence $\lambda_{k,\nu} \sim - \ii k \mu$. But we will still prove,
as claimed, that $|\lambda_{k,\nu}/k + \ii \mu|$ can be made as small as wanted.
We separate the two cases.

\begin{itemize}
 \item Let $\mu \in \Gamma_2$. From \cref{sec:construction2}, there exists 
 a smooth monotone $\Us$ such that $\Pinv(\mu) = 0$. Using the same arguments
 as above, for $\epsilon > 0$ small enough, $\Pinv$ maps small circles
 $\mu + \epsilon \ee^{\ii\theta}$ to (curves close to) small circles $c \epsilon^n \ee^{n\ii\theta}$.
 Once such an $\epsilon > 0$ is fixed, we define 
 $\gamma_+ := |c|\epsilon^n / 4$. Thus, the circles wind around the whole segment $[0,\gamma_+]$ with some positive minimal distance at least $2 \gamma_+$.
 For $k \geqslant K := (C/\gamma_+)^2$, thanks to \cref{thm:convergence-approx-inv},
 $\Phi(\mu+\epsilon \ee^{\ii\theta})$ also winds around the whole segment
 $[0,\gamma_+]$ and hence, for any $\gamma := (\sqrt{\nu}k)^{-1} \leqslant \gamma_+$
 there exists an eigenvalue $\lambda_{\nu,k} := - \ii k \mu_{k,\nu}$
 where $\mu_{k,\nu}$ is at a distance at most $\epsilon$ from $\mu$.
 
 \item Let $\mu \in \Gamma_3$. From \cref{sec:construction3}, for any $\gamma > 0$, there exists a smooth monotone $\Us$ such that $\Pinv(\mu) = \gamma$. Proceeding
 likewise yields circles winding around a small segment $[\gamma_-,\gamma_+]$
 and completes in a similar way the proof of \cref{thm:construction}.
\end{itemize}

\subsection{Almost explicit unstable shear flows}
\label{ssec:crit2-spec}

As announced in \cref{rk:examples}, we exhibit a family of shear flows 
satisfying \cref{crit:pdt-sufficient} and \cref{crit:ibl-global-sufficient}. 
For this, we consider a shear flow with $\Us''(0) > 0$ and $\Us''$
vanishing only once for some $y_0 > 0$. In order for the shear flow to
satisfy both \cref{crit:pdt-sufficient} and \cref{crit:ibl-global-sufficient},
it is sufficient to ensure that the corresponding crossing abscissa
$\chi(y_0)$ (see \eqref{eq:chi}) satisfies $\chi(y_0)<0$.
We define
\begin{equation*}
 \hs(u) = - \log(1-u) - \alpha \frac{u^2}{2}
\end{equation*}
which, for $\alpha \in [0,4)$, defines (the inverse of) a monotonic 
exponentially decaying profile. Indeed
\begin{equation*}
\hs'(u) = \frac{1}{1-u} - \alpha u= \frac{\alpha u^2 - \alpha u +1}{1-u}
\end{equation*}
which remains positive as long as $\alpha < 4$. We then find
\begin{equation*}
\hs''(u) = \frac{1}{(1-u)^2} - \alpha
\end{equation*}
so that for $\alpha > 1$ we have $\hs''(0) < 0$ and we have the desired
behavior. Then $\hs''(u_0) = 0$ holds for
\begin{equation*}
\alpha = \frac{1}{(1-u_0)^2}.
\end{equation*}
For the crossing abscissa at $y_0 := H_s(u_0)$, we now find
\begin{equation} \label{eq:chi-u0}
 \chi(y_0) = 
\frac{\hs'(0)}{u_0}
- \PV \int_0^1 \frac{(1-u)^2}{u-u_0} \left(\frac{1}{(1-u)^2} -
\alpha\right) \dd u
= \frac{1}{u_0} + \frac{u_0-\frac{3}{2}}{(1-u_0)^2}.
\end{equation}
Thanks to \eqref{eq:chi-u0} we can consider the condition $\chi(y_0) < 0$
(in order to match \cref{crit:ibl-global-sufficient}) or the
condition $\chi(y_0) < \ds$, where $\ds = 1-\alpha/6$ 
(in order to match \cref{crit:pdt-sufficient}.
The first condition holds for $u_0 > (7-\sqrt{17})/8$
and thus for $\alpha > 64/(1+\sqrt{17})^2 \approx 2.44$.
The second condition holds for $u_0^3 - 4 u_0^2 + \frac{13}{2}u_0 - 1 > 0$.
This cubic polynomial has a single root around 
$u_0 \approx 0.17$. 
Thus the condition holds for $\alpha > \alpha_\text{min} \approx 1.45$.

\section*{Acknowledgements}

The authors warmly thank Pierre-Yves Lagr\'ee for fruitful discussions and  explanations on the different boundary layer models.
\smallskip

\noindent A.-L.D. and D.G.-V. are partially supported by the Agence Nationale de la Recherche, project Dyficolti, Grant ANR-13-BS01-0003-01. 
D.G.-V. also acknowledges the support of the Institut Universitaire de France. H.D. acknowledges support of the Universit\'e Sorbonne Paris Cit\'e, in the framework of the ``Investissements d'Avenir'', convention ANR-11-IDEX-0005, and the People Programme (Marie Curie
Actions) of the European Union’s Seventh Framework Programme
(FP7/2007-2013) under REA grant agreement n. PCOFUND-GA-2013-609102,
through the PRESTIGE programme coordinated by Campus France.    
This project has received funding from the European Research Council (ERC) under the European Union's Horizon 2020 research and innovation program Grant agreement No 637653, project BLOC ``Mathematical Study of Boundary Layers in Oceanic Motion''.

\printbibliography

\appendix

\section{Appendix: Proofs of lemmas on \texorpdfstring{$\Pinv$}{the inviscid condition}}

\subsection*{Proof of  \cref{thm:possible-mu}}

\paragraph{Holomorphy of $\Pinv$} 
From the assumptions of \cref{sec:us-assumptions}, the range of $\Us$ 
 is $[0,1]$ and $F \in L^1$. 
 Then, we show that the complex derivative of $\Pinv$ exists by 
 the dominated convergence theorem and the compact support of $\Us$. 

\paragraph{Range of possible eigenvalues} 
Using integration by parts, we transform \eqref{eq:phi_inv} as
\begin{equation} \label{eq:phi-inv-raw}
  \Pinv(\mu) = - \int_0^\infty (1-U_s) \frac{(1+U_s-2\mu)(U_s-\conj{\mu})^2}{|U_s-\mu|^4}.
\end{equation}
The imaginary part of the integrand can be expressed using
\begin{equation} \label{eq:phi-inv-im}
 \Im\Big((1+U_s-2\mu)(U_s-\conj{\mu})^2\Big)
 = 2b \Big(b^2 + (U_s - a)(1-a) \Big).
\end{equation}
Since $b > 0$, the integral \eqref{eq:phi-inv-raw} can only be real if the imaginary
part of the integrand changes sign when $\Us$ goes from $0$ to $1$. Thanks to
\eqref{eq:phi-inv-im}, this can only happen if $b^2 < a(1-a)$. Conversely, if 
$b^2 \geqslant a(1-a)$ then $\Im \Pinv(\mu) < 0$.
\qed

\subsection*{Proof of the Plemelj formula (\cref{thm:plemelj})}
\textbf{First step. We express $\Pinv$ as a singular integral.}
 For $\mu = a + \ii b \not \in [0,1]$, 
 we integrate by parts~\eqref{eq:phi_inv} as follows
 \begin{equation} 
  \label{eq:phi-inv-ibp}
  \begin{split}
  \Pinv(\mu)
  & = \int_0^\infty \frac{F}{\Us'}
  \left(\frac{-\Us' (1-\mu)}{(\Us-\mu)^2}\right) \\
  & = \int_0^\infty \frac{F}{\Us'}
  \left(
  \frac{1-\mu}{\Us-\mu} - 1
  \right)' \\
  &= \left[
  \frac{F}{\Us'}
  \left(\frac{1-\mu}{\Us-\mu}
  - 1 \right)
  \right]_0^\infty
  - \int_0^\infty
  \left(
  \frac{F}{\Us'}
  \right)'
  \left(\frac{1-\mu}{\Us-\mu}
  - 1 \right).
  \end{split}
 \end{equation}
 Since $F(0) = 1$, the boundary term at $0$ is
 \begin{equation} \label{eq:phi-inv-ibp0}
  - \frac{F(0)}{\Us'(0)} \left(\frac{1-\mu}{\Us(0)-\mu} - 1 \right)
  = \frac{1}{\mu} \frac{1}{\Us'(0)}.
 \end{equation}
 Concerning the boundary term at infinity, we rewrite it using the identity
 \begin{equation} \label{eq:phi-inv-ibp-infty}
  \frac{F(z)}{\Us'(z)} \left(\frac{1-\mu}{\Us(z)-\mu} - 1 \right)
  = \frac{1}{\Us(z)-\mu} \left( \frac{(1-\Us(z))^2}{\Us'(z)} + (1-\Us(z))z \right).
 \end{equation}
 Thanks to assumptions~\eqref{eq:hyp-us-ratio1} and~\eqref{eq:hyp-us-ratio2},
 \eqref{eq:phi-inv-ibp-infty} tends to zero as $z \to +\infty$.
 Hence, the boundary term at infinity vanishes. Eventually, gathering
 \eqref{eq:phi-inv-ibp} and \eqref{eq:phi-inv-ibp0}, we find
 \begin{equation} \label{eq:phi-inv-ibp-u''}
  \Pinv(\mu)
  = \frac{1}{\mu \Us'(0)} 
  + \int_0^\infty
  \frac{(1-\Us(z))^2}{\Us'(z)^2} 
  \frac{\Us''(z)}{\Us(z)-\mu} \dd z,
 \end{equation}
 where integral converges thanks to assumption
 \eqref{eq:hyp-us-ratio3}. Using the monotonicity of the shear flow, we can use $u := U_s(z)$ 
 as an integration variable in \eqref{eq:phi-inv-ibp-u''}. Recalling \eqref{eq:def-g}, we obtain
 \begin{equation} \label{eq:phi-inv-mu-sing}
  \Pinv(\mu) = \frac{1}{\mu \Us'(0)} + \int_0^1 \frac{g(u)}{u-\mu} \dd u. 
 \end{equation}
 By Plemelj formula (see \cite[Chapter 1]{lu1994boundary}), one obtains that $\Pinv(\mu) \to G(a)$
 as $\mu \to a$. Although this statement is quite classical, we give a proof
 below both for the sake of completeness and because we want a uniform 
 convergence rate with respect to $a$ despite the fact that $g'$ is unbounded 
 near $u \approx 1$.
 
 \bigskip \noindent 
 \textbf{Second step. We estimate the rate of convergence.} 
 We start by introducing a smooth even function $\theta : \R \to \R_+$ 
 with $\theta(0) = 1$ and supported in $[-1,1]$. Recall that $0 < a_0 < a_1 < 1$. 
 We let $\tilde{a}_1 := (1+a_1)/2$
 and define $A := \max \{2/(1-a_1), 1/a_0\}$.
 Let $a \in [a_0, a_1]$. We define $\tilde{\theta}(u) := \theta(A(u-a))$
 which is supported in $(0,\tilde{a}_1)$ and bounded in $C^1$ uniformly with respect
 to $a$. On the one hand, since $\theta$ is even, one has
 \begin{equation} \label{eq:pv-1}
   \PV \int_0^1 \frac{g(u)}{u-a} \dd u
   = \int_0^1 \frac{g(u)-g(a)}{u-a} \tilde{\theta}(u) \dd u
   + \int_0^1 \frac{1-\tilde{\theta}(u)}{u-a} g(u) \dd u.
 \end{equation}
 On the other hand, letting $\mu := a + \ii b$ where $b \in (0,1]$, one has
 \begin{equation} \label{eq:pv-2}
  \begin{split}
   \int_0^1 \frac{g(u)}{u-\mu} \dd u
   = \int_0^1 \frac{g(u)-g(a)}{u-\mu} \tilde{\theta}(u) \dd u
   & + \int_0^1 \frac{1-\tilde{\theta}(u)}{u-\mu} g(u) \dd u \\
   & + g(a) \int_0^1 \frac{\tilde{\theta}(u)}{u-\mu} \dd u.
  \end{split}
 \end{equation}
 The first two terms of \eqref{eq:pv-2} converge to the right-hand side
 of \eqref{eq:pv-1}, while the third term yields the imaginary residue.
 We prove the convergence term by term. 
 
 \bigskip \noindent
 \textit{First term}. We use a splitting between close and far values.
 \begin{equation} \label{eq:pv-3}
  \begin{split}
   & \bigg| 
    \int_0^1 \frac{g(u)-g(a)}{u-\mu} \tilde{\theta}(u) \dd u
    - \int_0^1 \frac{g(u)-g(a)}{u-a} \tilde{\theta}(u) \dd u
   \bigg| \\
   & =
   \left| 
    \int_0^1 \frac{g(u)-g(a)}{u-a} \tilde{\theta}(u) \frac{\ii b}{u-\mu} \dd u
   \right| \\
   & \leqslant
    \int_{|u-a|\leqslant\sqrt{b}} \left| \frac{g(u)-g(a)}{u-a} \tilde{\theta}(u) \right| \dd u 
    +
    \sqrt{b} \int_{|u-a|\geqslant\sqrt{b}} \left| \frac{g(u)-g(a)}{u-a} \tilde{\theta}(u) \right| \dd u \\
   & \leqslant
   3 \sqrt{b} \| g' \|_{L^\infty(0,\tilde{a}_1)} \| \theta \|_{L^\infty}.
  \end{split}
 \end{equation}
 \textit{Second term}. We use the same splitting but exclude values $u > \tilde{a}_1$.
 \begin{equation} \label{eq:pv-4}
  \begin{split}
   & \bigg| 
    \int_0^1 \frac{1-\tilde{\theta}(u)}{u-\mu} g(u) \dd u
    - \int_0^1 \frac{1-\tilde{\theta}(u)}{u-a} g(u) \dd u
   \bigg| \\
   & =
   \left| 
    \int_0^1 \frac{1-\tilde{\theta}(u)}{u-a} g(u) \frac{\ii b}{u-\mu} \dd u
   \right| \\
   & \leqslant
   \left| 
    \int_0^{\tilde{a}_1} \frac{1-\tilde{\theta}(u)}{u-a} g(u) \frac{\ii b}{u-\mu} \dd u
   \right| +
   A^2 b \int_{\tilde{a}_1}^1 \left| (1-\tilde{\theta}(u)) g(u) \right| \dd u \\
   & \leqslant
   3 A \sqrt{b} \| \theta' \|_{L^\infty} \| g \|_{L^\infty(0,\tilde{a}_1)}
   + A^2 b (1+\|\theta\|_{L^\infty}) \| g \|_{L^1(0,1)}. 
  \end{split}
 \end{equation}
 \textit{Third term}. We start with the real part
 \begin{equation} \label{eq:re.trho}
   \Re \int_0^1 \frac{\tilde{\theta}(u)}{u-\mu} \dd u
   = \int_0^1 \frac{\tilde{\theta}(u)(u-a)}{(u-a)^2 + b^2} \dd u
   = 0,
 \end{equation}
 because $\theta$ is even and the support of $\tilde{\theta}$ is fully included 
 in $(0,1)$. Then, using that $\tilde{\theta}(a) = \theta(0) = 1$, we compute
 \begin{equation} \label{eq:im.trho}
  \begin{split}
   \Im \int_0^1 \frac{\tilde{\theta}(u)}{u-\mu} \dd u
   & = \int_0^1 \frac{\tilde{\theta}(u)b}{(u-a)^2 + b^2} \dd u \\
   & = \int_{-a/b}^{(1-a)/b} \frac{\tilde{\theta}(a+bs)}{1+s^2} \dd s \\
   & = \int_{-a/b}^{(1-a)/b} \frac{\dd s}{1+s^2} 
   + \int_{-a/b}^{(1-a)/b} \frac{\tilde{\theta}(a+bs)-\tilde{\theta}(a)}{1+s^2} \dd s.
  \end{split}
 \end{equation}
 Using \eqref{eq:re.trho} and estimating the integrals in \eqref{eq:im.trho}, 
 we deduce that there exists a constant $c_1 > 0$ such that
 \begin{equation} \label{eq:pv-5}
  \left| \ii \pi - \int_0^1 \frac{\tilde{\theta}(u)}{u-\mu} \dd u \right| 
  \leqslant
  \frac{b}{1-a_1} + \frac{b}{a_0} + A \| \theta' \|_{L^\infty} b (c_1 + |{\ln b}|).
 \end{equation}
 \textit{Conclusion}. Last we estimate the difference between the 
 inverses as
 \begin{equation} \label{eq:pv-6}
  \left| \frac{1}{\mu} - \frac{1}{a} \right| 
  \leqslant \frac{b}{a^2} \leqslant \frac{b}{a_0^2} \leqslant A^2 b.
 \end{equation}
 Gathering \eqref{eq:pv-3}, \eqref{eq:pv-4}, \eqref{eq:pv-5}
 and \eqref{eq:pv-6}
 proves the main estimate \eqref{eq:plemelj}.\qed

\begin{remark} \label{rk:plemelj-strong}
 It is clear from the proof of \cref{thm:plemelj} that we can actually
 state a stronger version highlighting the dependency on $a_0$ of
 the convergence. Indeed, we have proved that, for any $a_1 \in (0,1)$,
 there exists $C, \rho > 0$ such that, for any $a_0 \in (0,a_1)$
 and any $b \in (0,\rho]$, there holds
 \begin{equation} \label{eq:plemelj-strong}
  \left| \Pinv(a+\ii b) - G(a) \right| 
  \leqslant C \left( \frac{\sqrt{b}}{a_0} + \frac{b}{a_0^2} \right).
 \end{equation}
\end{remark}

\subsection*{Proof of \cref{thm:phi-inv}}

 \paragraph{Behavior of $\Pinv$ near 0}
 
 \emph{As a first step, we investigate the behavior of the imaginary part.}
 Let $a \in [0,\frac{1}{4}]$, $b \in (0,1]$ and $\mu := a + \ii b$. Using \eqref{eq:phi-inv-mu-sing}, we write
 \begin{equation} \label{eq:im-pinv-mu}
  \Im \Pinv(\mu) = \Im \frac{1}{\mu \Us'(0)} + g(a)\, \Im \int_0^1 \frac{\dd u}{u - \mu} +
  \Im \int_0^1 \frac{g(u)-g(a)}{u-\mu} \dd u.
 \end{equation}
 We estimate the second integral in \eqref{eq:im-pinv-mu} by splitting $[0,1]$ into  $[0,\frac{1}{2}]$ and $[\frac{1}{2},1]$.
 We find that there exists $c_1 > 0$ such that
 \begin{equation}
  \begin{split}
   \bigg| \Im & \int_0^1 \frac{g(u)-g(a)}{u-\mu} \dd u \bigg|
   = \left| \int_0^1 \frac{g(u)-g(a)}{(u-a)^2+b^2}\, b\, \dd u\right| \\
   & \leqslant \|g'\|_{L^\infty(0,\frac{1}{2})} \int_0^{\frac{1}{2}} \frac{|u-a|\, b\, \dd u}{(u-a)^2 + b^2} 
   + \frac{ b (\|g\|_{L^1(0,1)} + \|g\|_{L^\infty(0,\frac12)})}{(1/4)^2+b^2} \\
   & \leqslant b (c_1 + |{\ln b}|) \|g'\|_{L^\infty(0,\frac{1}{2})} + c_1 b \|g\|_{L^1(0,1)}.
  \end{split}
 \end{equation}
 The first integral in \eqref{eq:im-pinv-mu} can be computed explicitly. Indeed
 \begin{equation}
  \Im \int_0^1 \frac{\dd u}{u - \mu} = \pi - \arctan \frac{b}{1-a} - \arctan \frac{b}{a}.
 \end{equation}
 Hence there exists $M > 0$ such that, for $b$ small enough
 \begin{equation}
  \left| \Im \Pinv(\mu) + \frac{b}{\Us'(0)(a^2+b^2)}
  - g(a) \left(\pi - \arctan \frac{b}{a}\right) \right|
  \leqslant M \sqrt{b}.
 \end{equation}
As $U''_s(0)$ is non-zero, $g(0)$ is non-zero as well. Thus, for $0 < c_1 < \pi |g(0)| /2$, there exists $c_{\pm} > 0$ such that, 
 for $a, b$ small enough, 
 \begin{equation} \label{eq:im-pinv-smaller}
  | \Im \Pinv(a+\ii b) | \leqslant c_1
  \Longleftrightarrow
  c_- \leqslant \frac{b}{a^2+b^2} \leqslant c_+.
 \end{equation}
 
 \bigskip \noindent
 \emph{As a second step, we compute the real part under the
 condition \eqref{eq:im-pinv-smaller}}. We obtain
 \begin{equation}
  \Re \int_0^1 \frac{g(u)}{u-\mu} \dd u
  = g(a) \int_0^1 \frac{u-a}{(u-a)^2+b^2}\dd u + \int_0^1 \frac{(g(u)-g(a))(u-a)}{(u-a)^2+b^2} \dd u.
 \end{equation}
 The first integral is equal to
 \begin{equation}
  g(a) \int_0^1 \frac{u-a}{(u-a)^2+b^2} \dd u
  = \frac{g(a)}{2} \ln \left(\frac{b^2+(1-a)^2}{b^2+a^2}\right).
 \end{equation}
 The second integral can be bounded as
 \begin{equation}
  \left| \int_0^1 \frac{(g(u)-g(a))(u-a)}{(u-a)^2+b^2} \dd u \right| 
  \leqslant \frac{1}{2}\|g'\|_{L^\infty(0,\frac{1}{2})} 
  + 16 (\|g\|_{L^1(\frac{1}{2},1)} + \|g\|_{L^\infty(0,\frac{1}{2})}).
 \end{equation}
 Hence there exists $C > 0$ such that
 \begin{equation}
  \Re \Pinv(\mu) 
  \geqslant \frac{a}{\Us'(0)(a^2+b^2)} + \frac{g(a)}{2} \ln \left(\frac{b^2+(1-a)^2}{b^2+a^2}\right) - C.
 \end{equation}
 Let us assume that we are in a situation when both sides of
 \eqref{eq:im-pinv-smaller} are satisfied. For $b$ small enough,
 this implies that $b/(2c_+) \leqslant a^2 \leqslant 2b/c_-$.
 Since $g$ is continuous near $0$, for $b$ small enough,
 we have the estimate
 \begin{equation}
  \left| \frac{g(a)}{2} \ln \left(\frac{b^2+(1-a)^2}{b^2+a^2}\right) \right|
  \leqslant 2 |g(0)| |{\ln b}|.
 \end{equation}
 Moreover, still for $b$ small enough, we obtain
 \begin{equation}
  \frac{a}{\Us'(0)(a^2+b^2)} \geqslant
  b^{-\frac{1}{2}} \min (\sqrt{c_+}, \sqrt{c_-}/2 ).
 \end{equation}
 Gathering these estimates proves that there exists $c > 0$ such
 that, for $a, b$ small enough, $\Re \Pinv(a+\ii b) \geqslant c/\sqrt{b}$
 when \eqref{eq:im-pinv-smaller} holds.
 
 \bigskip \noindent
 \emph{Behavior above the circle}. We consider
 the case when $b^2 \geqslant a(1-a)$.
 Using \eqref{eq:phi-inv-mu-sing}, we compute
 \begin{equation}
  \Im \Pinv(\mu) = - \frac{b}{\Us'(0)(a^2+b^2)} + \int_0^1 \frac{g(u) b}{(u-a)^2+b^2} \dd u.
 \end{equation}
 If $a$ and $b$ are small enough and $b^2 > a(1-a)$, 
 then $b/(a^2+b^2) \geqslant 1/(2b)$. Moreover $\Us'(0) > 0$
 and the integral can be bounded. Indeed
 \begin{equation}
  \left| \int_0^1 \frac{g(u) b}{(u-a)^2+b^2} \dd u \right|
  \leqslant
  \pi \| g \|_{L^\infty(0,\frac{1}{2})} + 16 b \|g\|_{L^1(\frac{1}{2},1)}.
 \end{equation}
 Hence, for $a,b$ small enough above the circle, $\Im \Pinv(a+\ii b) \leqslant -c$ for some $c > 0$.

 \paragraph{Behavior of $\Pinv$ near 1}
 
 Assumption \eqref{eq:hyp-kappa-Us} implies that 
 \begin{equation} \label{eq:hyp-kappa-g}
 g(u) \leqslant - c_\kappa (1-u)^{\kappa}
 \quad \textrm{for} \quad u \geqslant u_\kappa := \Us(y_\kappa).
 \end{equation}
 We fix $\rho_1 := (1-u_\kappa) / 2$ and consider some 
 $a \in [1-\rho_1, 1]$. We compute
 \begin{equation}
 \Im \int_0^1 \frac{g(u)}{u-\mu} \dd u
 = b \int_0^{u_\kappa} \frac{g(u)}{(u-a)^2+b^2} \dd u + b \int_{u_\kappa}^1 \frac{g(u)}{(u-a)^2+b^2} \dd u
 \end{equation}
 The first integral is bounded by $b \rho_1^{-2} \|g\|_{L^1(0,u_\kappa)}$.
 Thanks to \eqref{eq:hyp-kappa-g},  
 the second integral is estimated as follows:
 \begin{equation}
 \begin{split}
 b \int_{u_\kappa}^1 \frac{g(u)}{(u-a)^2+b^2} \dd u
 & \leqslant - c_\kappa b \int_{a-\rho_1}^{a} \frac{(1-u)^\kappa}{(u-a)^2+b^2} \dd u \\
 & \leqslant - c_\kappa b^{\kappa} \int_{-\rho_1/b}^{0} \frac{(\frac{1-a}{b}-s)^\kappa}{1+s^2} \dd s \\
 & \leqslant - c_\kappa b^{\kappa} \int_{0}^{1} \frac{s^\kappa}{1+s^2} \dd s,
 \end{split}
 \end{equation}
 where we assumed in the last line that $b \leqslant \rho_1$.
 Using \eqref{eq:phi-inv-mu-sing}, we have
 \begin{equation}
 \Im \Pinv(\mu)
 \leqslant - \frac{b}{\Us'(0)(a^2+b^2)} 
  - c_\kappa b^{\kappa} \int_{0}^{1} \frac{s^\kappa}{1+s^2} \dd s
  + b \rho_1^{-2} \|g\|_{L^1(0,u_\kappa)}.
 \end{equation}
 In particular, there exists $c > 0$ small enough, 
 and $0 < \rho < \rho_1$ small enough
 such that $\Im \Pinv(a+\ii b) \leqslant - c b^\kappa$
 for $a \in [1-\rho,1]$ and $b \in (0,\rho]$.
 \qed

\subsection*{Proof of \cref{thm:properties-G}}

 Let $a \in (0,1/4]$. Using \eqref{eq:def-g}, we compute
 \begin{equation}
  \Re G(a) = \frac{1}{a \Us'(0)} + g(a) \PV \int_0^1 \frac{\dd u}{u-a}
  + \int_0^1 \frac{g(u)-g(a)}{u-a} \dd u.
 \end{equation}
 First,
 \begin{equation}
   g(a) \PV \int_0^1 \frac{\dd u}{u-a} = g(a) \int_{2a}^{1} \frac{\dd u}{u-a} 
   = g(a) \ln \left(\frac{1}{a}-1\right).
 \end{equation}
 Second,
 \begin{equation}
  \left|\int_0^1 \frac{g(u)-g(a)}{u-a} \dd u\right|
  \leqslant \frac{1}{2} \| g' \|_{L^\infty(0,\frac{1}{2})} + 4 \| g\|_{L^1(\frac{1}{2},1)}
  + 2 |g(a)|.
 \end{equation}
 Since $g$ is bounded near zero, we conclude that, for $c < 1/\Us'(0)$,
 there exists $\rho > 0$ such that $\Re G(a) \geqslant c / a$ 
 for $a \in (0,\rho]$.
\qed

\end{document}